\documentclass[a4paper, reqno, 10pt]{amsart}
\usepackage{amssymb}
\usepackage{extarrows}
\usepackage{bm}
\usepackage[centering]{geometry}
\usepackage{enumitem}
\usepackage{longtable}
\usepackage{multirow}
\usepackage{array}
\usepackage{etex}
\usepackage{pictex}
\usepackage{graphics}

\usepackage[all]{xy}

\numberwithin{equation}{section}
\theoremstyle{definition}
\newtheorem{defn}{Definition}[section]

\newtheorem{exm}[defn]{Example}

\theoremstyle{plain}
\newtheorem{cor}[defn]{Corollary}
\newtheorem{thm}[defn]{Theorem}
\newtheorem{lem}[defn]{Lemma}

\newtheorem{prop}[defn]{Proposition}
\newtheorem{fact}[defn]{Fact}
\newtheorem{defn-thm}[defn]{Definition-Theorem}

\newtheorem*{Freyd-Mitchell}{Freyd-Mitchell embedding Theorem}
\def\Fib{\operatorname{Fib}}
\def\CoFib{\operatorname{CoFib}}
\def\Weq{\operatorname{Weq}}
\def\Coker{\operatorname{Coker}}
\def\Hom{\operatorname{Hom}}
\def\Ext{\operatorname{Ext}}
\def\Ker{\operatorname{Ker}}
\def\Id{\operatorname{Id}}
\def\TFib{\operatorname{TFib}}
\def\TCoFib{\operatorname{TCoFib}}
\def\Ho{\operatorname{Ho}}
\def\X{\mathcal{X}}
\def\Y{\mathcal{Y}}
\def\A{\mathcal{A}}
\newcommand{\xra}{\xlongrightarrow}
\title[Model structure from one hereditary complete cotorsion pair]{Model structure  from \\ one hereditary complete cotorsion pair}
\author[Jian Cui, Xue-Song Lu, Pu Zhang] {Jian Cui, Xue-Song Lu, Pu Zhang$^*$ \\ \\  School of Mathematical Sciences \\
Shanghai Jiao Tong University,  \ Shanghai 200240, \ China }
\thanks{$^*$ Corresponding author}
\thanks{provinceanying$\symbol{64}$sjtu.edu.cn \ \ \ \ leocedar@sjtu.edu.cn \ \ \ \ pzhang$\symbol{64}$sjtu.edu.cn}
\thanks{This work is supported by the National Key Research and Development Program Grant 2022YFA1004900, by National Natural Science Foundation of China Grant 12131015, and by Natural Science Foundation of Shanghai Grant 23ZR1435100.}

\begin{document}
\maketitle
\begin{abstract} In contrast with the Hovey correspondence of abelian model structures from two complete cotorsion pairs,
Beligiannis and Reiten give a construction of model structures on abelian categories from one hereditary complete cotorsion pair.
The aim of this paper is to extend this result to weakly idempotent complete exact categories,
by adding the condition of heredity of the complete cotorsion pair. In fact, even for abelian  categories,
this condition of heredity should be added. This construction really gives model structures which are not necessarily exact in the sense of Gillespie.
The correspondence of Beligiannis and Reiten of weakly projective model structures also holds for weakly idempotent complete exact categories.

\vskip5pt

Keywords:   (weakly idempotent complete) exact category; (abelian, exact, weakly projective, projective) model structure; (complete, hereditary) cotorsion pair;
homotopy category

\vskip5pt

2020 Mathematics Subject Classification.   Primary 18N40, 16D90, 16E30; Secondary 16E65, 16G50, 16G20

\end{abstract}

\vskip10pt

\section {\bf Introduction}

The Hovey correspondence ([H2]) of abelian model structures gives an effective construction of model structures on abelian categories.
Exact category is an important generalization of abelian category: any full subcategory of an abelian category
which is closed under extensions and direct summands is a weakly idempotent complete exact category,
but not abelian in general. M. Hovey's correspondence has been extended as the one-one correspondence
between exact model structures and the Hovey triples on weakly idempotent complete exact categories, by
J. Gillespie [G] (see also J. \v{S}t'ov\'{\i}\v{c}ek \cite {S}).

\vskip5pt

A Hovey triple involves two complete cotorsion pairs. A. Beligiannis and I. Reiten give a construction of weakly projective model structures
([BR, VIII, 4.2, 4.13]) on abelian categories $\mathcal A$, from only one complete cotorsion pair.
These weakly projective model structures are different from abelian model structures, in general. The two approaches get the same result if and only if $\mathcal A$ has enough projective objects and
the model structure is {\it projective} in the sense of Gillespie \cite[4.5]{G}, i.e., it is abelian and each object is fibrant. For example, this is the case of the model structure induced by
the Gorenstein-projective modules over a Gorenstein algebra.

\vskip5pt The aim of this paper is to extend the results of Beligiannis and Reiten to weakly idempotent complete exact categories.

\subsection{The $\omega$-model structures}  \ Let $\mathcal{A}$ be a weakly idempotent complete exact category,
$\mathcal{X}$ and $\mathcal{Y}$ full additive subcategories of $\mathcal{A}$ which are closed under direct summands and isomorphisms. Put $\omega:=\mathcal{X}\cap \mathcal{Y}$.
As in [BR, VIII, 4] for abelian categories, consider the following construction.

\vskip5pt

Denote by ${\rm CoFib}_{\omega}$ the class of inflations $f$ with \ $\Coker f\in \mathcal{X}$.

\vskip5pt

Denote by ${\rm Fib}_{\omega}$ the class of morphisms $f: A\longrightarrow B$ such that $f$ is $\omega$-epic, i.e.,
$\text{Hom}_{\mathcal{A}}(W, f): \Hom_{\mathcal{A}}(W,A)\longrightarrow \text{Hom}_{\mathcal{A}}(W,B)$ is surjective, for any object $W\in \omega$.

\vskip5pt

Denote by ${\rm Weq}_{\omega}$ the class of morphisms $f: A\longrightarrow B$ such that there is a deflation $(f, t): A\oplus W \longrightarrow B$ with $W\in \omega$ and $\Ker (f, t)\in \mathcal{Y}$.
Thus, a morphism $f: A\longrightarrow B$ is in ${\rm Weq}_{\omega}$ if and only if there is a commutative diagram
\[\xymatrix@R=0.5cm{A\ar[dr]_-{\left(\begin{smallmatrix}1\\ 0  \end{smallmatrix}\right)}\ar[rr]^-f && B\\ & A\oplus W\ar[ur]_-{(f, t)}}\]
such that $W\in \omega$, \ $(f, t)$ is a deflation, and $\Ker (f, t)\in \mathcal{Y}$.

\begin{thm}\label{mainthm} \ \ {\rm (See Theorems \ref{ifpart} and \ref{onlyif})} \ Let $\mathcal{A}$ be a weakly idempotent complete exact category,
$\mathcal{X}$ and $\mathcal{Y}$ full additive subcategories of $\mathcal{A}$ which are closed under isomorphisms and direct summands, and $\omega =\mathcal{X}\cap \mathcal{Y}$. Then $({\rm CoFib}_{\omega}, \ {\rm Fib}_{\omega}, \ {\rm Weq}_{\omega})$ is a  model structure on $\A$ if and only if
$(\mathcal{X},\mathcal{Y})$ is a hereditary complete cotorsion pair in $\mathcal{A}$, and $\omega$ is contravariantly finite in $\mathcal A$.

\vskip5pt

If this is the case, then the class $\TCoFib_\omega$ of trivial cofibrations is precisely
the class of splitting monomorphisms with cokernel in $\omega$, and the class $\TFib_\omega$
of trivial fibrations is precisely the class of deflations with kernel in $\mathcal Y;$ the class of cofibrant objects of this model structure is $\mathcal X,$
the class of fibrant objects is $\mathcal A,$
and the class of trivial objects is $\mathcal Y;$ and the homotopy category
of this model structure is the additive quotient $\mathcal X/\omega$.
\end{thm}

\vskip5pt

For a full additive subcategory $\mathcal U$ of an additive category $\A$,
recall that quotient category $\mathcal A/\mathcal U$ has the same objects as $\mathcal A$, and
\[
\Hom_{\mathcal A/\mathcal U}(X,Y)=\Hom_{\mathcal A}(X,Y)/\Hom_{\mathcal A}(X, \mathcal U, Y)
\]
where $\Hom_{\mathcal A}(X, \mathcal U, Y)$ is the subgroup $\{f\in \Hom_{\mathcal A}(X,Y) \ | \ f \ \text{factors \ through\ an \ object \ in \ } \mathcal U\}$.
Then $\mathcal A/\mathcal U$ is an additive category.

\vskip5pt

The originality of Theorem \ref{mainthm} is due to A. Beligiannis and I. Reiten for abelian categories. See [BR, VIII, Theorem 4.2].
However, even for abelian categories, the original result Theorem 4.2 in [BR, VIII] misses the condition of the heredity of the complete cotorsion pair $(\mathcal{X},\mathcal{Y})$:
an example shows that if the heredity is {\bf not} required, then Theorem 4.2 in [BR, VIII] does {\bf not} hold.
See Proposition \ref{resolvingcoresolving} and Example \ref{nothereditary}.

\vskip5pt

The proof  of Theorem \ref{mainthm} is essentially different from the one of Theorem 4.2 in [BR, VIII] for abelian categories:
the two out of three axiom and of the retract axiom are the most difficult parts in the proof,
and our proofs for these two parts are more direct, avoiding using stabilizations and left triangulated categories as in [BR, VIII, Lemma 4.1].

\vskip5pt

The model structure in Theorem \ref{mainthm} is called the {\it $\omega$-model structure} ([BR]). Recall that a model structure on an exact category is {\it exact} ([G, 3.1]), if
cofibrations are precisely inflations with cofibrant cokernel, and fibrations are precisely
deflations with fibrant kernel. The Hovey correspondence gives a one-one correspondence between exact model structures and the Hovey triples, on
a weakly idempotent complete exact category.  The connection and difference between the $\omega$-model structures and the abelian model structures on an abelian category
is clear by [BR, VIII, 4.13]. Also, the connection and difference between $\omega$-model structures and the exact model structures on a weakly idempotent complete exact category
is clear as follows. This  $\omega$-model structure is exact if and only if $\mathcal A$ has enough projective objects
and $\omega = \mathcal P$, the class of projective objects of $\mathcal A$. See Proposition \ref{thesame}.
Thus, this $\omega$-model structure is {\bf not} an exact model structure,  i.e., it can {\bf not} be obtained by the Hovey triples via the Hovey correspondence, in general.
In fact, using the hereditary complete cotorsion pairs induced by tilting objects in exact categories ([Kr]), one gets
$\omega$-model structures which are not exact, even on weakly idempotent complete exact categories which are not abelian. See Examples \ref{newmodel} and \ref{generalexm}.

\subsection{The heredity} \begin{prop} \label{resolvingcoresolving}   \ Let $\mathcal A$ be a weakly idempotent complete exact category,
$(\mathcal{X},\mathcal{Y})$  a complete cotorsion pair, and $\omega = \X\cap \Y$.
If $(\CoFib_{\omega}, \Fib_{\omega}, \Weq_{\omega})$ is a model structure, then the cotorsion pair $(\mathcal{X}, \mathcal{Y})$ is hereditary.
\end{prop}
\begin{proof} \ It suffices to prove that $\Y$ is closed under the cokernels of inflations
(see Lemma \ref{hereditary}). Suppose that there is an admissible exact sequence
    \[
    \xymatrix{
    0\ar[r]& Y_1\ar[r]& Y_2\ar[r]^-{d} & C\ar[r]& 0
    }
    \]
with  $Y_i\in \mathcal{Y}$ for $i = 1, 2.$  By the construction the morphism $0: Y_2\longrightarrow 0$ is in $\Weq_{\omega}$, since
$(0, 0): Y_2\oplus 0\longrightarrow 0$ is a deflation with $0\in \omega$ and $\Ker (0, 0) = Y_2\in\Y$.
In a similar way, $d: Y_2\longrightarrow C$ is in $\Weq_{\omega}$, since
$(d, 0): Y_2\oplus 0\longrightarrow C$ is a deflation with $\Ker (d, 0) = Y_1\in\Y$.

\vskip5pt

Since $(Y_2\longrightarrow 0) = (C\longrightarrow 0) \circ d$,  by the two out of three axiom
 the morphism $0: C\longrightarrow 0$ is in $\Weq_{\omega}$. By definition there is a deflation $0:  C\oplus W\longrightarrow 0$
with  $W\in \omega$ and $C\oplus W\in \mathcal{Y}$. Thus $C\in \mathcal  Y$.
\end{proof}

\subsection{The correspondence of Beligiannis and Reiten}

A model structure on an exact category is {\it weakly projective} if
cofibrations are precisely inflations with cofibrant cokernel, each trivial fibration is a deflation, and each object is fibrant.
This is equivalent to say that trivial fibrations are precisely deflations with trivially fibrant kernel, each cofibration is an inflation, and each object is fibrant.
More equivalent characterizations of  a weakly projective model structure on a weakly idempotent complete exact category
are given in Proposition \ref{weaklyproj}. As in abelian categories (\cite[VIII, 4.6]{BR}), the $\omega$-model structures on a weakly idempotent complete exact category are precisely
the weakly projective model structures.

\begin{thm} \label{brcorrespondence} \ {\rm (The correspondence of Beligiannis and Reiten)} \ Let $\A$ be a weakly idempotent complete exact category, $S_C$ the class of hereditary complete cotorsion pairs $(\X,\Y)$ with $\omega=\X\cap \Y$ contravariantly finite, and $S_M$ the class of weakly projective model structures on $\A$. Then the maps $\Phi: (\X,\Y)\mapsto (\CoFib_{\omega}, \Fib_{\omega}, \Weq_{\omega})$ and $\Psi: (\CoFib, \Fib, \Weq)\mapsto (\mathcal{C}, \rm T\mathcal{F})$ give a bijection between $S_C$ and $S_M$, where $\mathcal C$ is the class of cofibrant objects, and $\rm T\mathcal{F}$ is the class of trivially fibrant objects.
\end{thm}

Thus, the intersection of the class of Hovey's exact model structures and the class of  Beligiannis and Reiten's $\omega$-model structures, on  a weakly idempotent complete exact category, is exactly
the classes of projective model structures, . See Subsection 5.3.

\subsection{The organization} Section 2 recalls  necessary preliminaries on (weakly idempotent complete) exact categories, including the Extension-Lifting Lemma,
(hereditary complete) cotorsion pairs, model structures and the homotopy categories, the Hovey correspondence of exact model structures.

\vskip5pt

Section 3 is devoted to the proof of the ``if" part of Theorem \ref{mainthm}.
An example of a complete cotorsion pair which is not hereditary with core $\omega$ contravariantly finite is given, and hence
$({\rm CoFib}_{\omega}, \ {\rm Fib}_{\omega}, \ {\rm Weq}_{\omega})$ is not a model structure.
This $\omega$-model structure is exact if and only if $\mathcal A$ has enough projective objects
and $\omega$ is the class of projective objects; thus it gives model structures which are not necessarily exact.

\vskip5pt

Section 4 is to prove the ``only if" part of Theorem \ref{mainthm}. In Section 5, weakly projective model structures are characterized, and Theorem \ref{brcorrespondence} is proved.
Finally, the dual version of Theorems \ref{mainthm} and \ref{brcorrespondence} is stated in Subsection 5.4.

\section{\bf Preliminaries}

\subsection{Exact categories}

Let $\mathcal A$ be an additive category. An {\it exact pair} $(i, d)$ is a sequence of morphisms
$X \xlongrightarrow{i} Y \xlongrightarrow{d} Z$ in $\mathcal A$ such that $i$ is a kernel of $d$, and $d$ is a cokernel of $i$.
Two exact pairs $(i, d)$ and  $(i', d')$ is  {\it isomorphic} if there is a commutative diagram
$$\xymatrix@R=0.4cm{X \ar[r]^-i \ar[d] & Y \ar[d] \ar[r]^-{d} & Z\ar[d] \\
X' \ar[r]^-{i'} & Y'\ar[r]^-{d'} & Z'}$$
such that all the vertical morphisms are isomorphisms.  The following definition given by B. Keller is equivalent to the original one in D. Quillen {\cite[\S2]{Q3}}.

\begin{defn} (\cite [Appendix A]{Kel}) \label{exactcat} An exact category is a pair $(\mathcal A, \mathcal E)$, where $\mathcal A$ is an additive category,
and $\mathcal E$ is a class of exact pairs satisfying the axioms (E0), (E1), (E2) and (E2$^{\rm op}$), where
an exact pair $(i, d)\in \mathcal E$ is called a {\it conflation}, $i$  an {\it inflation}, and $d$  a {\it deflation}.

\vskip5pt

(E0) \ $\mathcal E$ is closed under isomorphisms, and  ${\rm Id}_0$ is a deflation.

\vskip5pt

(E1) \ The composition of two deflations is a deflation.

\vskip5pt

(E2) \ For any deflation $d: Y\longrightarrow Z$ and any morphism $f: Z'\longrightarrow Z$, there is a pullback
\begin{equation}\label{pb}\xymatrix@R=0.4cm{Y' \ar@{.>}[r]^-{d'}\ar@{.>}[d]_{f'} & Z' \ar[d]^-f \\
Y \ar[r]^-d & Z}\end{equation}
such that $d'$ is a deflation.

\vskip5pt

(E2$^{\rm op}$) \  For any inflation $i: X\longrightarrow Y$ and any morphism $f: X\longrightarrow X'$,  there is a pushout
\begin{equation}\label{po}\xymatrix@R=0.4cm{X \ar[r]^-i \ar[d]_-f & Y \ar@{.>}[d]^-{f'} \\
X' \ar@{.>}[r]^-{i'} & Y'}\end{equation}
such that $i'$ is an inflation.
\end{defn}

A sequence $0\longrightarrow X\stackrel i\longrightarrow Y\stackrel d\longrightarrow Z\longrightarrow 0$ of morphisms in exact category
$\mathcal A$ is {\it an admissible exact sequence } if $(i, d)$ is a conflation.

\begin{fact} \label{factex}  Let $\mathcal A$ be an exact category. Then

{\rm (1)} \ The composition of inflations is an inflation.

{\rm (2)} \ An isomorphism is a deflation and an inflation$;$ a deflation which is monic is an isomorphism$;$ an inflation which is epic is an isomorphism.

$(3)$ \ For any objects $X$ and $Y$, $Y\xra{\binom{0}{1}} X\oplus Y \xra{(1, 0)} X$ is a conflation.

$(4)$ \ Let $(f, g)$ and $(f', g')$ be conflations.  Then the direct sum $(\left(\begin{smallmatrix} f&0\\0&f'\end{smallmatrix}\right), \left(\begin{smallmatrix} g&0\\0&g'\end{smallmatrix}\right))$  is a conflation.

$(5)$ \ Let $i: A\longrightarrow B$ be an inflation, $a: A\longrightarrow X$ be an arbitrary morphism. Then $\binom{i}{a}: A\longrightarrow B\oplus X$ is an inflation.
Let $j: A\longrightarrow B$ be a deflation, $b: X\longrightarrow B$ be an arbitrary morphism. Then $(j, b): A\oplus X\longrightarrow B$ is a deflation.

$(6)$ \ Let $i: A\longrightarrow B$ and $p: B\longrightarrow A$ such that $pi=1_A$.  Then $i$ is an inflation if and only if $p$ is a deflation.\end{fact}

\begin{lem} {\rm (\cite[2.15]{Bu})} \label{specialpp}  \ Let $\mathcal A$ be an exact category.

\vskip5pt

$(1)$ \ Let {\rm(\ref{pb})} be a pullback with $d$ a deflation.
If $f$ is an inflation, then so is $f'$.

\vskip5pt

$(1')$ \ Let {\rm(\ref{po})} be a pushout with $i$ an inflation.
If $f$ is a deflation, then so is $f'$.
\end{lem}

\begin{lem} {\rm (\cite[2.19]{Bu})} \label{inflation}  \ Let $\mathcal A$ be an exact category.

\vskip5pt

$(1)$ \ Let {\rm(\ref{pb})} be a pullback such that $d$ and $f'$ are deflations. Then $f$ is a deflation.

\vskip5pt

$(1')$ \ Let {\rm(\ref{po})} be a pushout such that $i$ and $f'$ are inflations. Then $f$ is an inflation.
\end{lem}

For the assertion $(1')$ below, $i'$ is assumed to be an inflation in \cite{Bu}. However, this assumption can be removed.

\begin{lem} {\rm (\cite[2.12]{Bu})} \label{pp} \ Let $\mathcal A$ be an exact category.

\vskip5pt

${\rm (1)}$ \ Consider the commutative square in $\mathcal A$
$$\xymatrix@R=.4cm{B' \ar[r]^-{d'} \ar[d]_-{f'} & C' \ar[d]^-{f} \\
                      B \ar[r]^-{d} & C}$$
with deflation $d$. Then the following are equivalent.

\vskip5pt

\hskip20pt {\rm (i)} \ It is a pullback.

\vskip5pt

\hskip20pt {\rm (ii)} \ The sequence $0\longrightarrow B'\xlongrightarrow{\left(\begin{smallmatrix} d' \\ -f' \end{smallmatrix}\right)} C'\oplus B \xlongrightarrow{(f, d)} C\longrightarrow 0$ is admissible exact.

\vskip5pt

\hskip20pt {\rm (iii)} \ It is both a pullback and a pushout.

\vskip5pt

\hskip20pt {\rm (iv)} \ There is a commutative diagram with admissible exact rows
$$\xymatrix@R=.4cm{0 \ar[r] & A \ar[r]\ar@{=}[d] & B' \ar[r]^-{d'}\ar[d]_-{f'} & C' \ar[r] \ar[d]^-f & 0 \\
                   0 \ar[r] & A \ar[r]& B \ar[r]^-{d}         & C \ar[r] & 0.}$$

\vskip10pt

${\rm (1')}$ \  Consider the commutative square in $\mathcal A$
$$\xymatrix@R=.4cm{A \ar[r]^-i \ar[d]_-f & B \ar[d]^-{f'} \\
                      A' \ar[r]^-{i'} & B'}$$
with inflation $i$. Then the following are equivalent.

\vskip5pt

\hskip20pt {\rm (i')} \ It is a pushout.

\vskip5pt

\hskip20pt {\rm (ii')} \ The sequence $0\longrightarrow A\xlongrightarrow{\left(\begin{smallmatrix} i \\ -f \end{smallmatrix}\right)} B\oplus A' \xlongrightarrow{(f', i')} B'\longrightarrow 0$ is admissible exact.

\vskip5pt

\hskip20pt {\rm (iii')} \ It is both a pushout and a pullback.

\vskip5pt

\hskip20pt {\rm (iv')} \ There is a commutative diagram with admissible exact rows
$$\xymatrix@R=.4cm{0 \ar[r] & A \ar[r]^-i\ar[d]_-f & B \ar[r]\ar[d]_-{f'} & C \ar[r]\ar@{=}[d] & 0 \\
                   0 \ar[r] & A' \ar[r]^-{i'}      & B' \ar[r]         & C \ar[r] & 0.}$$
 \end{lem}

\vskip5pt

We need the following facts. Under the assumption of weakly idempotent completeness,
they are corollaries of \cite [8.11] {Bu}. For the convenience we drop the assumption.

\begin{lem}\label{twodeflations} \ Let $\alpha: A \longrightarrow B$ and $\beta: B\longrightarrow C$ be morphisms in an exact category $\mathcal A$.
\vskip5pt

$(1)$ \ If $\alpha$ and $\beta$ are deflations, then there is an admissible exact sequence
$0\longrightarrow \Ker \alpha \longrightarrow \Ker \beta\alpha \longrightarrow \Ker \beta \longrightarrow 0$ in $\mathcal A.$

\vskip5pt

$(1')$ \ If $\alpha$ and $\beta$ are inflations, then there is an admissible exact sequence
$0\longrightarrow \Coker \alpha \longrightarrow \Coker \beta\alpha \longrightarrow \Coker \beta \longrightarrow 0$ in $\mathcal A.$

\vskip5pt

$(2)$ \ If $\alpha$ is an inflation, $\beta\alpha$ is a deflation, then $\beta$ is a deflation and there is an admissible exact sequence
$0\longrightarrow \Ker \beta\alpha \longrightarrow \Ker \beta \longrightarrow \Coker \alpha \longrightarrow 0$ in $\mathcal A.$

\vskip5pt

$(2')$ \ If $\beta$ is a deflation, $\beta\alpha$ is an inflation, then $\alpha$ is an inflation and there is an admissible exact sequence
$0\longrightarrow \Ker \beta \longrightarrow \Coker \alpha \longrightarrow \Coker \beta\alpha \longrightarrow 0$ in $\mathcal A.$

\end{lem}
\begin{proof} \ By duality we only prove (1) and (2).

\vskip5pt

(1) \ There is a commutative diagram with admissible exact sequences in rows:

\[\xymatrix@R=0.6cm{
0 \ar[r]& \Ker \beta\alpha\ar[r]\ar@{..>}[d]_\gamma &A\ar[d]^-{\alpha}\ar[r]^-{\beta\alpha}& C\ar[r]\ar@{=}[d]& 0\\
0 \ar[r]& \Ker \beta\ar[r] & B\ar[r]^-{\beta} & C\ar[r]& 0}
\]
By Lemma \ref{pp}$(1')$ the left square is both a pushout and a pullback.
Since $\alpha$ is a deflation, $\gamma$ is a deflation.
Lemma \ref{pp}(1) gives a commutative diagram with admissible exact rows and columns

\[
\xymatrix@R=0.4cm{
& 0\ar[d] & 0\ar[d] & &\\
& \Ker \alpha\ar@{=}[r]\ar[d] & \Ker \alpha\ar[d] & &\\
0 \ar[r]& \Ker \beta\alpha\ar[r]\ar@{..>}[d]_\gamma &A\ar[d]^-{\alpha}\ar[r]^-{\beta\alpha}& C\ar[r]\ar@{=}[d]& 0\\
0 \ar[r]& \Ker \beta\ar[r]\ar[d] & B\ar[d]\ar[r]^-{\beta} & C\ar[r]& 0.\\
& 0 & 0 &&  \\
}
\]

\vskip5pt

(2) \ Consider the pushout of $\alpha$ and $\beta\alpha$. Then there is a commutative diagram
\[
\xymatrix@R=0.4cm{A\ar[r]^-{\alpha}\ar[d]_-{\beta\alpha}& B\ar[d]^-{\phi}\ar@/^/[rdd]^-{\beta}&\\
C\ar@{=}@/_/[rrd]\ar[r]^-{\gamma}&E\ar@{..>}[rd]^-{t}&\\
&&C
}
\]
with inflation $\gamma$. By Lemma \ref{factex}(6), $t$ is a splitting deflation. By Lemma \ref{specialpp}$(1')$, $\phi$ is a deflation. Thus $\beta= t\phi$ is a deflation. Now there is a morphism $\delta$ such that the diagram commutes:
 \[\xymatrix@R=0.4cm{
0 \ar[r]& \Ker \beta\alpha\ar[r]\ar@{..>}[d]_\delta &A\ar[d]^-{\alpha}\ar[r]^-{\beta\alpha}& C\ar[r]\ar@{=}[d]& 0\\
0 \ar[r]& \Ker \beta\ar[r] & B\ar[r]^-{\beta} & C\ar[r]& 0}
\]
By Lemma \ref{pp}$(1')$ the left square is a pushout. By Lemma \ref{inflation}$(1')$,
$\delta$ is an inflation. Then by Lemma \ref{pp}$(1')$ one gets the desired admissible exact sequence.
\end{proof}

\subsection{Extension-Lifting Lemma} The Extension-Lifting Lemma will play an important role.
It has been proved for abelian categories in [BR, VIII, 3.1], and for exact categories in \cite[5.14]{S}.

\vskip5pt

\begin{lem} \label{extlifting} \ Let $\mathcal A$ be an exact category and $X,Y \in \mathcal{A}$.
Then $\Ext^1_{\mathcal{A}} (X,Y) = 0$ if and only if for any commutative diagram with $(i, d)$ and $(c, p)$ conflations
$$  \xymatrix@R=0.4cm{& 0 \ar[r] & A \ar[r]^-i\ar[d]_-{\alpha} & B\ar[r]^-d\ar[d]^-\beta &X\ar[r]& 0 \\
    0 \ar[r] & Y\ar[r]^-c & C\ar[r]^-p & D \ar[r] & 0}$$
there exists a morphism $\lambda: B \longrightarrow C$ such that $\alpha = \lambda i$ and $\beta = p \lambda$.
\end{lem}
\begin{proof} For convenience we include a slightly different proof for ``the only if" part. Assume that $\Ext^1_{\mathcal{A}} (X,Y)= 0$. For any commutative diagram above with conflations $(i, d)$ and $(c, p)$,
making the pullback of $p$ and $\beta$,  by Lemma \ref{pp}(1) there is a commutative diagram
    $$\xymatrix@R=0.4cm{
        0\ar[r] & Y\ar[r]^{\varepsilon}\ar@{=}[d] & K\ar[r]^{\zeta}\ar[d]^{\gamma} & B\ar[r]\ar[d]^{\beta}& 0 \\
        0\ar[r] & Y\ar[r]^c & C\ar[r]^p & D\ar[r] & 0.
    }$$
Since $p\alpha = \beta i$, there is a unique morphism $\phi : A \longrightarrow K$ such that $i=\zeta \phi$ and $\alpha = \gamma \phi $.
Since $i=\zeta \phi$ is an inflation and $\zeta$ is a deflation, $\phi$ is an inflation by Lemma \ref{twodeflations}$(2')$, say with deflation $\xi: K \longrightarrow L$.
Since $i=\zeta \phi$, there is a commutative diagram
    $$\xymatrix@R=0.4cm{
        0\ar[r] & A\ar[r]^{\phi}\ar@{=}[d] & K\ar[r]^{\xi}\ar[d]^{\zeta} & L\ar[r]\ar@{..>}[d]^{\eta}& 0 \\
        0\ar[r] & A\ar[r]^i & B\ar[r]^d & X\ar[r] & 0.
    }$$
By Lemma \ref{pp}(1) the right square above is a pullback. By Lemma \ref{inflation}(1), $\eta$ is a deflation. Then by Lemma \ref{pp}(1), $\Ker \eta \cong \Ker \zeta = Y$. Since $\Ext^1_{\mathcal{A}} (X,Y) = 0$, $\eta$ is a splitting deflation, thus $\zeta$ is also a splitting deflation. So, there is $g: B \longrightarrow K $ with $\zeta g = \Id_B$.
Then $p(\alpha - \gamma g i) = 0$. Thus there is $\mu: A \longrightarrow Y$ with $c \mu = \alpha - \gamma g i$.
 By exact sequence $\Hom_{\A}(B,Y)\longrightarrow\Hom_{\A}(A,Y)\longrightarrow\Ext^1_{\mathcal{A}} (X,Y) = 0$,
 there is $\nu: B\longrightarrow Y$ with $\nu i = \mu$. Then $\alpha = (c \nu + \gamma g) i.$
 Put $\lambda = c \nu + \gamma g$. Then $\alpha = \lambda i$ and \ $p \lambda = p\gamma g = \beta\zeta g = \beta.$ \end{proof}

\subsection{Weakly idempotent complete exact categories}

\begin{lem} \label{wicec} {\rm (\cite [Appendix]{DRSSK}; \cite[7.2, 7.6]{Bu})} \  Let $\mathcal A$ be an exact category. Then the following are equivalent$:$

{\rm (i)} \ Any splitting epimorphism in  $\mathcal A$ is a deflation.

{\rm (ii)} \ Any splitting epimorphism in  $\mathcal A$ has a kernel.

{\rm (iii)} \ Any splitting monomorphism in  $\mathcal A$ is an inflation.

{\rm (iv)} \ Any splitting monomorphism in  $\mathcal A$ has a cokernel.

{\rm (v)} \ If $de$ is a deflation, then so is $d$.

{\rm (vi)} \ If $ki$ is an inflation, then so is $i$.
\end{lem}

An exact category satisfying the above equivalent conditions in Lemma \ref{wicec} is called {\it a  weakly idempotent complete} exact category (\cite{Bu}; \cite[1.11.5]{TT}).

\subsection{Cotorsion pairs in exact categories}

Let $\mathcal A$ be an exact category, $\mathcal C$ a class of objects of  $\mathcal A$. Define
\ ${}^\perp\mathcal C = \{X\in \A \ | \ \Ext_\A^1(X, C) = 0, \ \forall \ C\in \mathcal C \}$ and \ $\mathcal C^{\perp}= \{ Y\in \A \ |\ \Ext_\A^1(C, Y)=0, \ \forall \  C\in \mathcal C\}.$
A pair \ $(\mathcal C, \ \mathcal F)$ of classes of objects of  $\mathcal A$ is
a {\it cotorsion pair}, if \ $\mathcal C={}^\perp\mathcal F$ \ and \
$\mathcal F = \mathcal C^{\perp}.$
A cotorsion pair $(\mathcal C, \mathcal F)$  is  {\it complete}, if for any object $X\in \mathcal A$, there are admissible exact sequences
$$0\longrightarrow F\longrightarrow C\longrightarrow X\longrightarrow 0, \quad \text{and}\quad
0\longrightarrow X\longrightarrow F'\longrightarrow C'\longrightarrow 0,$$
with $C, \ C'\in \mathcal C$, and \ $F, \ F'\in \mathcal F$.

\vskip5pt

A cotorsion pair $(\mathcal C, \mathcal F)$  is  {\it hereditary}, if $\mathcal C$ is closed under the kernel of deflations, and $\mathcal F$ is closed under the cokernel of inflations.

\begin{lem}\label{hereditary} \ {\rm (\cite[6.17]{S})} \ Let $(\mathcal C, \mathcal F)$ be a complete cotorsion pair in a weakly idempotent complete exact category $\A$.
Then the following are equivalent$:$

\vskip5pt

$(1)$ \ $(\mathcal C, \mathcal F)$ is hereditary$;$

\vskip5pt

$(2)$ \ $\mathcal C$ is closed under the kernel of deflations$;$

\vskip5pt

$(3)$  \  $\mathcal F$ is closed under the cokernel of inflations$;$

\vskip5pt

$(4)$ \ $\Ext^2_\A(\mathcal C, \mathcal F)=0$ for all $C\in \mathcal C$ and $F\in \mathcal F;$

\vskip5pt

$(5)$  \ $\Ext^i_\A(\mathcal C, \mathcal F)=0$ for all $C\in \mathcal C$, $F\in \mathcal F$, and $i\geq 2$.
\end{lem}

\subsection{\bf Model structures}
\begin{defn} \label{closedms} \  {\rm ([Q1], [Q2])} \ A closed model structure on a category $\mathcal M$ is a triple
\ $(\CoFib$, \ $\Fib$,  \ $\Weq)$ of classes of morphisms,
where the morphisms in the three classes are respectively called {\it cofibrations, fibrations, and weak equivalences},
satisfying the following axioms:

\vskip5pt

{\bf Two out of three axiom} \ Let $X\xlongrightarrow{f}Y\xlongrightarrow{g}Z$ be morphisms in $\mathcal M$. If two of the morphisms $f, \ g, \ gf$ are weak equivalences,
then so is the third one.

\vskip5pt

{\bf Retract axiom}  \  If $g$ is a retract of $f$, and $f$ is a cofibration (a fibration, a weak equivalence, respectively), then so is $g$.

\vskip5pt

{\bf Lifting axiom} \  Cofibrations have the left lifting property with respect to all morphisms in $\Fib\cap \Weq$, and
fibrations have the right lifting property with respect to all the morphisms in $\CoFib\cap \Weq$. That is, given a commutative square
$$\xymatrix@R=0.4cm{A\ar[r]^-a \ar[d]_-i & X \ar[d]^-p \\
B\ar[r]^-b \ar@{.>}[ru]^-s & Y }$$
with $i\in \CoFib$ and $p\in \Fib$, if either $i\in \Weq$ or $p\in \Weq$,
then there exists a morphism $s: B\longrightarrow X$ such that $a = si$ and $b = ps$.

\vskip5pt

{\bf Factorization axiom}  \ Any morphism $f: X\longrightarrow Y$ admits factorizations \ $f=pi$ and  \ $f=qj$, where \ $i\in \CoFib\cap \Weq$, \ $p\in \Fib$,  \ $j\in \CoFib$, and \ $q\in \Fib\cap \Weq$.
\end{defn}

The morphisms in $\CoFib\cap \Weq$ (respectively, $\Fib\cap \Weq$)
are called {\it trivial cofibrations} (respectively, {\it trivial fibrations}). Put
$\TCoFib: = \CoFib\cap \Weq$  and
$\TFib: = \Fib\cap \Weq$.

\vskip5pt

Following [H1] (also [Hir]), we will call a closed model structure just as {\it a model structure}. But then a model structure here
is different from  a ``model structure" in the sense of [Q1]:
it is a ``model structure" in [Q1], but the converse is not true (see [Q1], pages 5.1 - 5.2; and Proposition 2 at page 5.5). The following facts are in the axioms of a ``model structure" in [Q1], Thus one has

\begin{fact} \label{elementpropmodel} \ Let  $(\CoFib$, \ $\Fib$,  \ $\Weq)$ be a
model structure on category $\mathcal M$ with zero object. Then

$(1)$ \ Both the classes $\CoFib$ and $\Fib$ are closed under compositions.

$(2)$ \ Isomorphisms are fibrations, cofibrations, and weak equivalences.

$(3)$ \ Cofibrations are closed under pushouts, i.e., given a pushout square
$$\xymatrix@R=0.4cm{\bullet\ar[r]^-i\ar[d] & \bullet \ar@{.>}[d] \\
\bullet \ar@{.>}[r]^-{i'} & \bullet}
$$
with $i\in \CoFib$, then $i'\in \CoFib$.

Also, trivial cofibrations are closed under pushouts.

$(4)$ \ Fibrations are closed under pullbacks$;$ and trivial fibrations are closed under pullbacks.\end{fact}

For a model structure $(\CoFib$, \ $\Fib$,  \ $\Weq)$ on category $\mathcal M$ with zero object,
an object $X$ is {\it trivial} if $0 \longrightarrow X $ is a weak equivalence, or, equivalently,
$X\longrightarrow 0$ is a weak equivalence.
It is {\it cofibrant} if $0\longrightarrow X$ is a cofibration, and it is {\it fibrant} if $X\longrightarrow 0$ is a fibration.
An object is {\it trivially cofibrant} (respectively, {\it trivially fibrant}) if it is both trivial and  cofibrant (respectively, fibrant).

\vskip5pt

A striking property of a model structure is that any two classes of $\CoFib$, \  $\Fib$, \  $\Weq$ uniquely determine the third.

\vskip5pt

\begin{prop}\label{quillenlifting} {\rm ([Q2, p.234])} \ Let $(\CoFib, \ \Fib,  \ \Weq)$ be a model structure on category $\mathcal M$. Then

\vskip5pt

$(1)$ \ Cofibrations are precisely those morphisms which have the left lifting property with respect to all the trivial fibrations.

\vskip5pt

$(2)$ \ Trivial cofibrations are precisely those morphisms which have the left lifting property  with respect to all the fibrations.

\vskip5pt

$(3)$ \ Fibrations are precisely those morphisms which have the right lifting property with respect to all the trivial cofibrations.

\vskip5pt

$(4)$ \ Trivial fibrations are precisely those morphisms which have the right lifting property with respect to all the cofibrations.

\vskip5pt

$(5)$ \ $\Weq = \TFib\circ \TCoFib.$

\end{prop}

\subsection{Quillen's homotopy category} For a model structure  on category $\mathcal M$ with zero object, Quillen's {\it homotopy category} is the localization $\mathcal M[\Weq^{-1}]$, and is denoted by ${\rm Ho}(\mathcal M)$.

\vskip5pt

Let $(\CoFib$, \ $\Fib$,  \ $\Weq)$ be a model structure on category $\mathcal M$ with zero object,  finite coproducts and finite products, 
such that there exist push-outs of two trivial cofibrations and pull-backs of two trivial fibrations.  
Let $\mathcal{M}_{cf}$ be the full subcategory of $\mathcal{M}$ consisting of all the cofibrant and fibrant objects.
Recall from [Q1] that the left homotopy relation $\overset{l}\sim$ coincides with the right homotopy relation $\overset{r}\sim$ \ in $\mathcal{M}_{cf}$, which is denoted by  $\sim$
(see Lemma 5 and its dual on p. 1.8 in [Q1]). Then $\sim$ is an equivalence relation of $\mathcal{M}_{cf}$, and the corresponding quotient category is denoted by $\pi\mathcal{M}_{cf}$:
the objects are the same as the ones of $\mathcal{M}_{cf}$, and the morphism set is
$\pi(A, B)$, the set of equivalence classes of $\Hom_\mathcal M(A, B)$ respect to the relation
 $\sim$. By Theorem 1' in [Q1, p. 1.13], the composition of the embedding
 $\mathcal{M}_{cf}\hookrightarrow \mathcal{M}$ and the localization functor $\mathcal{M} \longrightarrow {\rm Ho}(\mathcal{M})$ induces an equivalence
\ $\pi\mathcal{M}_{cf} \longrightarrow {\rm Ho}(\mathcal{M})$ of categories.

\subsection{\bf The Hovey correspondence} \ A model structure on an exact category is {\it exact} ([G, 3.1]), if
cofibrations are exactly inflations with cofibrant cokernel, and fibrations are exactly
deflations with fibrant kernel. In this case, trivial cofibrations are exactly inflations with trivially cofibrant cokernel, and trivial fibrations are exactly
deflations with trivially fibrant kernel. If $\mathcal A$ is an abelian category, then an exact model structure on $\mathcal A$ is just an abelian model structure in [H2].

\vskip5pt

{\it A Hovey triple} in an exact category is a triple \ $(\mathcal C, \mathcal F, \mathcal W)$ \ of classes of objects such that
\  $\mathcal W$ is {\it thick}, i.e., $\mathcal W$ is closed under direct summands,
and if two out of three terms in an admissible exact sequence  are in $\mathcal W$, then so is the third one; and that both $(\mathcal C \cap \mathcal W, \ \mathcal F)$ and \ $(\mathcal C, \ \mathcal F \cap \mathcal W)$ \  are complete cotorsion pairs.

\begin{thm} \label{hoveycorrespondence} {\rm (The Hovey correspondence) \ ([G, 3.3]; \cite[6.9]{S}; see also [H2, Theorem 2.2])} \ Let $\mathcal A$ be a weakly idempotent complete exact category.
Then there is a one-to-one correspondence between exact model structures and the Hovey triples in $\mathcal A$, given by
$$({\rm CoFib}, \ {\rm Fib}, \ {\rm Weq})\mapsto (\mathcal{C}, \ \mathcal{F}, \ \mathcal W)$$
where \ $\mathcal C  = \{\mbox{cofibrant objects}\}, \ \
\mathcal F  = \{\mbox{fibrant objects}\}, \ \
\mathcal W  = \{\mbox {trivial objects}\}$, with the inverse \ \ $(\mathcal{C}, \ \mathcal{F}, \ \mathcal W) \mapsto ({\rm CoFib}, \ {\rm Fib}, \ {\rm Weq}),$ where
\begin{align*} &{\rm CoFib} = \{\mbox{inflations with cokernel in} \ \mathcal{C}\}, \ \ \
{\rm Fib}  = \{\mbox{deflations with kernel in} \ \mathcal{F} \}, \\
& {\rm Weq}  = \{pi \ \mid \ i \ \mbox{is an inflation,} \ \Coker i\in \mathcal{C}\cap \mathcal W, \ p \ \mbox{is a deflation,} \ \Ker p\in \mathcal{F}\cap \mathcal W\}.\end{align*}
\end{thm}

\section{\bf Model structure induced by a hereditary complete cotorsion pair}

The aim of this section is to prove the ``if" part of Theorem \ref{mainthm}, namely

\begin{thm}\label{ifpart} \ Let $\mathcal A$ be a weakly idempotent complete exact category.
If $(\mathcal{X},\mathcal{Y})$ is a hereditary complete cotorsion pair in $\mathcal{A}$ such that the core $\omega=\mathcal{X}\cap \mathcal{Y}$ is contravariantly finite in $\mathcal A$.
Then $({\rm CoFib}_{\omega}, \ {\rm Fib}_{\omega}, \ {\rm Weq}_{\omega})$ is a  model structure,  the class $\TCoFib_\omega$ of trivial cofibrations is precisely
the class of splitting monomorphisms with cokernel in $\omega$, and the class $\TFib_\omega$
of trivial fibrations is precisely the class of deflations with kernel in $\mathcal Y.$ \end{thm}

\subsection{Descriptions of \ ${\rm CoFib}_{\omega} \cap \Weq_{\omega}$ and ${\rm Fib}_{\omega}\cap \Weq_{\omega}$}  \ As in [BR, VIII, 4] for abelian categories, put
\begin{align*}{\rm TCoFib}_{\omega} & = \{\mbox{splitting monomorphism} \ f \ | \ \Coker f\in \omega\} \\
{\rm TFib}_{\omega} & = \{\text{deflation} \ f \ |  \ \Ker f\in \mathcal{Y}\}.\end{align*}
Note that any morphism in ${\rm TCoFib}_{\omega}$ is an inflation and that $\Weq_{\omega}$ can be reformulated as
$$\Weq_{\omega}=\{gf \ | \ f\in \TCoFib_{\omega}, \ g\in \TFib_{\omega}\} = {\rm TFib}_{\omega}\circ {\rm TCoFib}_{\omega} .$$

\vskip5pt

The following fact will be important in the proof later, and it is less clear.

\begin{lem} \label{tcofibandtfib} \ Let $\mathcal A$ be a weakly idempotent complete exact category, $\mathcal{X}$ and $\mathcal{Y}$ full additive subcategories of $\mathcal{A}$ which are closed under isomorphisms and direct summands, and $\omega=\mathcal{X}\cap \mathcal{Y}$. If $\Ext^1_\mathcal A(\X, \Y) = 0$. Then
$${\rm TCoFib}_{\omega} = {\rm CoFib}_{\omega} \cap \Weq_{\omega}, \ \ \ \ \ \ \ {\rm TFib}_{\omega} = {\rm Fib}_{\omega}\cap \Weq_{\omega}.$$
\end{lem}
\begin{proof} \ We first prove ${\rm TCoFib}_{\omega} = {\rm CoFib}_{\omega} \cap \Weq_{\omega}$.  \ Let $f\in {\rm TCoFib}_{\omega}.$ That is, $f$ is a splitting monomorphism with $\Coker f\in \omega$.

\vskip5pt

Clearly $f\in {\rm CoFib}_{\omega}$. Without loss of generality one may assume that $f$ is just $f=\left(\begin{smallmatrix} 1\\0\end{smallmatrix}\right): A\longrightarrow A\oplus W$ where $W\in \omega$.
By the definition one sees $f\in \Weq_{\omega}$, by taking $t=\left(\begin{smallmatrix} 0\\1\end{smallmatrix}\right): W\longrightarrow A\oplus W$.
Conversely, let $f: A\longrightarrow B\in {\rm CoFib}_{\omega}\cap {\rm Weq}_{\omega}$. By definition $f$ is an inflation with $\Coker f\in \X$ and there is an admissible exact sequence
\[
\xymatrix{
0\ar[r]& Y\ar[r]& A\oplus W\ar[r]^-{(f,t)}& B\ar[r]&0
}
\]
with $W\in \omega$ and $Y\in \mathcal{Y}$. Since $\Coker f\in \X$ and $\Ext^1_\mathcal A(\X, \Y) = 0$, by the Extension-Lifting Lemma \ref{extlifting}, there is a lifting $\left(\begin{smallmatrix} u\\v \end{smallmatrix}\right): B\longrightarrow A\oplus W$
such that $\left(\begin{smallmatrix} u\\v \end{smallmatrix}\right)f=\left(\begin{smallmatrix} 1\\0 \end{smallmatrix}\right)$ and  $(f,t)\left(\begin{smallmatrix} u\\v \end{smallmatrix}\right)=1_B$.
See the diagram
\[
    \xymatrix@R=0.5cm{
    & 0\ar[r]& A\ar[r]^-{f}\ar[d]_-{\left(\begin{smallmatrix} 1 \\ 0 \end{smallmatrix}\right)} & B \ar[r]^-{p}\ar@{=}[d]\ar@{..>}[ld]_-{(u,v)} &\Coker f\ar[r]&0\\
    0\ar[r]& Y\ar[r]&A\oplus W\ar[r]^-{(f,t)}& B\ar[r]&0.
    }
    \]
 Thus $f$ is a splitting inflation. Moreover, there is a morphism $\gamma: W\longrightarrow \Coker f$ making the diagram commute
\[
\xymatrix@R=0.5cm{
0\ar[r]& A\ar[r]^-{\left(\begin{smallmatrix} 1\\0 \end{smallmatrix}\right)}\ar@{=}[d]& A\oplus W\ar[r]^-{(0,1)}\ar[d]^-{(f,t)}& W\ar@{..>}[d]^-{\gamma}\ar[r]&0\\
0\ar[r]& A\ar[r]^-{f}& B\ar[r]^-{p}& \Coker f\ar[r]&0.
}
\]
By Lemma \ref{pp}(1), the right square above is a pullback. Since $p$ and $(f, t)$ are deflations in this pullback square, it follows from
Lemma \ref{inflation}(1) that $\gamma$ is a deflation. By Lemma \ref{pp}(1),  $\Ker \gamma=\Ker (f,t)\in \Y$.
Since $\Ext^1_\mathcal A(\X, \Y) = 0$, by the admissible exact sequence $0\longrightarrow \Ker \gamma\longrightarrow W\longrightarrow \Coker f\longrightarrow 0$, one sees that $\Coker f$ is a summand of $W$. Thus
$\Coker f\in \omega$. By definition $f\in {\rm TCoFib}_{\omega}$. This proves ${\rm TCoFib}_{\omega} = {\rm CoFib}_{\omega} \cap \Weq_{\omega}$.

\vskip10pt

Next, we prove the second equality ${\rm TFib}_{\omega} = {\rm Fib}_{\omega}\cap \Weq_{\omega}.$ \ Let $f\in {\rm TFib}_{\omega},$ i.e., $f: A\longrightarrow B$ is a deflation with $\Ker f\in \mathcal{Y}$.
Since $0\in \omega$, it follows from the definition that $f\in \Weq_{\omega}$. For $W\in \omega$, it follows from $\Ext^1_\mathcal A(\X, \Y) = 0$ and the admissible exact sequence
$0\longrightarrow \Ker f \longrightarrow A \stackrel f\longrightarrow B \longrightarrow  0 $ that there is an exact sequence
\[
\xymatrix{
\text{Hom}_{\mathcal{A}}(W, A)\ar[r] &\text{Hom}_{\mathcal{A}}(W, B)\ar[r] &\text{Ext}_{\mathcal{A}}^1(W, \Ker f)=0.
}
\]
Thus $f$ is an $\omega$-epimorphism, i.e.,  $f\in {\rm Fib}_{\omega}$.

\vskip5pt

Conversely, let $f\in {\rm Fib}_{\omega}\cap {\rm Weq}_{\omega}$. Then there is an admissible exact sequence
\[
\xymatrix{
0\ar[r]& \Ker (f,t)\ar[r]& A\oplus W\ar[r]^-{(f,t)}& B\ar[r]&0
}
\]
with $W\in \omega$ and $\Ker(f,t)\in \mathcal{Y}$. Since $f$ is $\omega$-epic, there is some $s:W\longrightarrow A$ such that $t=fs$.
Then $(f,t)=f(1,s)$. Since  $\mathcal{A}$ is a weakly idempotent complete exact category and $(f, t)$ is a deflation, it follows that
$f$ is a deflation. Now there is a commutative diagram with admissible exact rows
\[
\xymatrix@R=0.5cm{
&0\ar[r]& \Ker f\ar[r]^-{\sigma}\ar@{..>}[d]_-{g}& A\ar[r]^-{f}\ar[d]_-{\left(\begin{smallmatrix}1\\0 \end{smallmatrix}\right)}&B\ar@{=}[d]\ar[r]&0\\
&0\ar[r]& \Ker(f,t)\ar[r]\ar@{..>}[d]_-{h}& A\oplus W\ar[r]^-{(f,t)}\ar[d]_{(1,s)}& B\ar@{=}[d] \ar[r]&0\\
&0\ar[r]& \Ker f\ar[r]^-{\sigma}& A\ar[r]^-{f}&B\ar[r]&0.\\
}
\]
Since $\sigma hg=\sigma$ and $\sigma$ is an inflation (thus a monomorphism), $hg={\rm Id}_{\Ker f}$.  Since  $\mathcal{A}$ is weakly idempotent complete, $g$ is an inflation.
Thus $\Ker f$ is a summand of $\Ker (f,t)\in \mathcal{Y}$, and hence $\Ker f\in \mathcal{Y}$. Thus $f\in {\rm TFib}_{\omega}$.
This completes the proof. \end{proof}

\subsection{Factorization axiom} We first prove the factorization axiom, i.e.,  every morphism $f:A \longrightarrow B$ can be factored as $f= pi$ with
$i \in {\rm CoFib}_{\omega} \cap \Weq_{\omega} = \TCoFib_{\omega}$ and \ $p \in \Fib_{\omega}$,  and
$f= q j$ with $j\in \CoFib_{\omega}$ and \ $q\in {\rm Fib}_{\omega}\cap \Weq_{\omega} = \TFib_{\omega}$. Here Lemma \ref{tcofibandtfib} has been already used.

\vskip5pt
\begin{lem}\label{TFib} \ $(1)$\ The class ${\rm CoFib}_\omega$ is closed under composition.
\vskip5pt

$(2)$\ The class ${\rm TFib}_\omega$ is closed under composition.
\end{lem}
\begin{proof}
(1) \ Let $\alpha: A \longrightarrow B$ and $\beta: B\longrightarrow C$ be in ${\rm CoFib}_\omega$.
Since $\beta\alpha$ is an inflation,
it suffices to show that $\Coker\beta\alpha\in \X$. By Lemma \ref{twodeflations}$(1')$ there is an admissible exact sequence
$0\longrightarrow \Coker \alpha \longrightarrow \Coker \beta\alpha \longrightarrow \Coker \beta \longrightarrow 0.$
Since $\Coker \alpha$ and $\Coker \beta$ are in $\X$,  $\Coker \beta\alpha \in \X$.
\vskip5pt
(2) \ Let $\alpha: A \longrightarrow B$ and $\beta: B\longrightarrow C$ be in ${\rm TFib}_\omega$.
Since $\beta\alpha$ is a deflation,
it suffices to show that $\Ker\beta\alpha\in \Y$. By Lemma \ref{twodeflations}(1) there is an admissible exact sequence
$0\longrightarrow \Ker \alpha \longrightarrow \Ker \beta\alpha \longrightarrow \Ker \beta \longrightarrow 0.$
Since $\Ker \alpha$ and $\Ker \beta$ are in $\Y$,   $\Ker \beta\alpha \in \Y$.
\end{proof}

\vskip5pt

{\bf The first factorization.} \ Since $\omega$ is contravariantly finite, there is a right $\omega$-approximation $\tau_B: T_B \longrightarrow B$.
Then $(f, \tau_B): A\oplus T_B \longrightarrow B$ is $\omega$-epic: in fact, for each morphism $g: W \longrightarrow B$ with $W \in \omega$, there is a morphism $h: W \longrightarrow T_B$ such that $g = \tau_B h$;
and then
$g = (f,\tau_B)\left(\begin{smallmatrix} 0 \\ h \end{smallmatrix}\right)$ with $\left(\begin{smallmatrix} 0 \\ h \end{smallmatrix}\right): W \longrightarrow A \oplus T_B$.

\vskip5pt

Thus one has the factorization $f = (f,\tau_B)\left(\begin{smallmatrix} 1 \\ 0 \end{smallmatrix}\right)$,
where $\left(\begin{smallmatrix} 1 \\ 0 \end{smallmatrix}\right):  A \longrightarrow A\oplus T_B$ is in $\TCoFib_{\omega}$ and $(f,\tau_B): A\oplus T_B \longrightarrow B$ is in $\Fib_{\omega}$.

\vskip5pt

{\bf The second factorization.} \ Taking an admissible exact sequence $0 \longrightarrow Y_B \longrightarrow X_B \overset{t_B}{\longrightarrow} B \longrightarrow 0$ of $B$  with $X_B\in \mathcal X$ and $Y_B\in \mathcal Y$,
one gets a deflation $(f,t_B): A \oplus X_B \longrightarrow B$, by Lemma \ref{factex}(6), say with the kernel $k: K \longrightarrow A \oplus X_B$.
Taking an admissible exact sequence $0 \longrightarrow K \overset{\sigma}{\longrightarrow} Y \longrightarrow X \longrightarrow 0$ of $K$
with $Y\in \mathcal Y$ and $X\in \mathcal X$, and forming the pushout of $k$ and $\sigma$,
one gets inflations $g$ and $i$, and a commutative diagram
$$\xymatrix@R=0.5cm{
     0\ar[r] & K\ar[r]^-k\ar[d]_\sigma & A \oplus X_B\ar[d]^i\ar[r]^-{(f,t_B)} & B\ar[r]\ar@{=}[d] & 0 \\
     0\ar[r] & Y\ar[r]^g & E\ar[r]^p & B\ar[r] & 0}$$
By Lemma \ref{pp}$(1')$ one has $\Coker i \cong \Coker \sigma = X\in \X$. By definition $i \in \CoFib_{\omega}$ and $p \in \TFib_{\omega}$.

\vskip5pt

Thus $f = p\circ (i\circ \left(\begin{smallmatrix} 1 \\ 0 \end{smallmatrix}\right)),$  where
$i \circ \left(\begin{smallmatrix} 1 \\ 0 \end{smallmatrix}\right): A \overset{\left(\begin{smallmatrix} 1 \\ 0 \end{smallmatrix}\right)}{\longrightarrow} A\oplus X_B \overset{i}{\longrightarrow} E$.
By Lemma \ref{TFib}(1) one has $i \circ \left(\begin{smallmatrix} 1 \\ 0 \end{smallmatrix}\right) \in \CoFib_{\omega}$. \hfill $\square$

\subsection{Two out of three axiom} The proof of the two out of three axiom is different from
the one for abelian categories in [BR, VIII, Theorem 4.2]. We do not use arguments in left triangulated categories.

\vskip5pt

\begin{lem}\label{23} \ Let $\alpha: A \longrightarrow B$ and $\beta: B\longrightarrow C$ be morphisms in $\mathcal A$.
If two of the three morphisms $\alpha, \ \beta, \ \beta\alpha$ are in ${\rm Weq}_\omega$, then so is the third.
\end{lem}

To prove Lemma \ref{23}, we need some preparations.

\begin{lem}\label{231} \ The class ${\rm Weq}_\omega$ is closed under compositions.
\end{lem}
\begin{proof} \ Let $\alpha: A \longrightarrow B$ and $\beta: B\longrightarrow C$ be in ${\rm Weq}_\omega$.
By definition, there is a morphism $(\alpha, t_1): A\oplus W_1 \longrightarrow B$ in ${\rm TFib}_\omega$
with $W_1\in \omega$, and a morphism $(\beta, t_2):  B\oplus W_2 \longrightarrow C$ in ${\rm TFib}_\omega$ with
$W_2\in \omega$. Then  $\beta\alpha$ has the decomposition of
$$\beta\alpha =(\beta\alpha, \beta t_1 , t_2)\left(\begin{smallmatrix} 1\\0\\0\end{smallmatrix}\right)$$
with $\left(\begin{smallmatrix} 1\\0\\0\end{smallmatrix}\right): A \longrightarrow A\oplus W_1\oplus W_2$ and $(\beta\alpha, \beta t_1 , t_2): A\oplus W_1\oplus W_2\longrightarrow C$. See the following diagram.

\[\xymatrix@R=0.5cm{
A\ar[dr]_-{\left(\begin{smallmatrix} 1\\ 0 \end{smallmatrix}\right)}\ar[rr]^-{\alpha}& & B\ar[dr]_-{\left(\begin{smallmatrix}1\\ 0   \end{smallmatrix}\right)}\ar[rr]^-{\beta}&&C\\
&A\oplus W_1\ar[dr]_(.4){\left(\begin{smallmatrix}1&0\\0& 1\\0&0   \end{smallmatrix}\right)}\ar[ur]_-{\left(\begin{smallmatrix} \alpha, t_1 \end{smallmatrix}\right)}&&B\oplus W_2\ar[ur]_-{\left(\begin{smallmatrix} \beta, t_2 \end{smallmatrix}\right)}&\\
&&A\oplus W_1\oplus W_2\ar[ur]_(.65){\left(\begin{smallmatrix}\alpha& t_1 & 0\\ 0 &0 &1   \end{smallmatrix}\right)} &&
}
\]

\vskip5pt

\noindent Since $(\alpha, t_1)\in {\rm TFib}_{\omega}$, it follows that $\left(\begin{smallmatrix}\alpha& t_1 & 0\\ 0 &0 &1   \end{smallmatrix}\right)\in {\rm TFib}_{\omega}$ by Fact \ref{factex}(4).
Thus $(\beta\alpha, \beta t_1 , t_2) = (\beta, t_2)\left(\begin{smallmatrix}\alpha& t_1 & 0\\ 0 &0 &1   \end{smallmatrix}\right)\in {\rm TFib}_\omega$,  by Lemma \ref{TFib}.
Hence $\beta\alpha=(\beta\alpha, \beta t_2 , t_1)\left(\begin{smallmatrix} 1\\0\\0\end{smallmatrix}\right)\in {\rm Weq}_\omega$.
\end{proof}

\begin{lem}\label{necessityofWeq} \ Let $\alpha: A\longrightarrow B$ be a morphism in ${\rm Weq}_\omega.$  Then for an arbitrary right $\omega$-approximation $t: W\longrightarrow B$ of $B$, the morphism \ $(\alpha, t): A\oplus W\longrightarrow B$ is in ${\rm TFib}_\omega$.
\end{lem}
\begin{proof}  By the assumption, there is a morphism $(\alpha, t'): A\oplus W'\longrightarrow B$ in ${\rm TFib}_\omega$
with $W'\in \omega$. Thus $\Ker(\alpha, t')\in \Y$.
Let $t: W\longrightarrow B$ be an arbitrary right $\omega$-approximation.
Then there is a morphism $s: W'\longrightarrow W$ such that $t'=ts$.
Since  $(\alpha, t')$ is a deflation and $(\alpha, t')=(\alpha, t)\left(\begin{smallmatrix} 1 & 0\\0 & s\end{smallmatrix}\right)$,
$(\alpha, t): A\oplus W \longrightarrow B$ is also a deflation.
It remains to prove that $\Ker (\alpha, t)\in\Y$.

\vskip5pt

Since $\left(\begin{smallmatrix} 1 & 0\\0 & 1\\ 0 & 0 \end{smallmatrix}\right): A\oplus W'\longrightarrow A\oplus W'\oplus W$ is a splitting monomorphism,
it is an inflation with cokernel $W$.  Since $\left(\begin{smallmatrix} 1 & 0\\0 & 1\\ 0 & 0 \end{smallmatrix}\right)$ is an inflation and
$(\alpha, t')=(\alpha, t', t)\left(\begin{smallmatrix} 1 & 0\\0 & 1\\ 0 & 0 \end{smallmatrix}\right)$ is a deflation, it follows from
Lemma \ref{twodeflations}(2) that
there is an admissible exact sequence
$$0\longrightarrow \Ker (\alpha, t') \longrightarrow \Ker (\alpha, t', t) \longrightarrow
\Coker \left(\begin{smallmatrix} 1 & 0\\0 & 1\\ 0 & 0 \end{smallmatrix}\right)\longrightarrow 0.$$
See the diagram below.
Since $\Ker(\alpha, t')\in \Y$ and $\Coker \left(\begin{smallmatrix} 1 & 0\\0 & 1\\ 0 & 0 \end{smallmatrix}\right) = W\in \mathcal{Y}$. Thus $\Ker (\alpha, t', t)\in \Y$.

\[
\xymatrix@R=0.5cm{
& 0\ar[d] & 0\ar[d] & &\\
0 \ar[r]& \Ker (\alpha,t')\ar[r]^-i\ar[d]&A\oplus W'\ar[d]_-{\left(\begin{smallmatrix}1 & 0\\0 &1\\0 & 0 \end{smallmatrix}\right)}\ar[r]^-{(\alpha,t') }& B\ar[r]\ar@{=}[d]& 0\\
0 \ar[r]& \Ker (\alpha,t',t)\ar[r]^-{i'}\ar[d] & A\oplus W'\oplus W\ar[d]\ar[r]^-{(\alpha,t',t)} & B\ar[r]& 0\\
& W\ar@{=}[r]\ar[d] & W\ar[d] & &\\
& 0 & 0 &&  \\
}
\]

\vskip5pt

By Fact \ref{factex}(5) and (4),  $\left(\begin{smallmatrix}1&0&0\\0&s&1 \end{smallmatrix}\right): A\oplus W'\oplus W\longrightarrow A\oplus W$ is a deflation, in fact it is a splitting deflation with kernel $W'$.
Since $(\alpha, t', t) = (\alpha, t)\left(\begin{smallmatrix}1&0&0\\0&s&1 \end{smallmatrix}\right)$, it follows from Lemma \ref{twodeflations}(1) that
there is an admissible exact sequence
$$0\longrightarrow \Ker \left(\begin{smallmatrix}1&0&0\\0&s&1 \end{smallmatrix}\right) \longrightarrow \Ker (\alpha, t', t) \longrightarrow
\Ker (\alpha, t)\longrightarrow 0.$$
Note that $\Ker \left(\begin{smallmatrix}1&0&0\\0&s&1 \end{smallmatrix}\right) = W'\in \Y$ and $\Ker (\alpha, t', t)\in \mathcal{Y}$.
Since by the assumption that $(\X, \Y)$ is a hereditary cotorsion pair,  $\Ker (\alpha, t)\in Y$. This completes the proof.
\end{proof}

\begin{lem}\label{TFibWeq} \ Let $\alpha: A\longrightarrow B$ and $\beta: B\longrightarrow C$ be morphisms in $\A$ with $\alpha\in {\rm TFib}_\omega$ and $\beta\alpha\in {\rm Weq}_\omega$. Then $\beta \in {\rm Weq}_\omega$.
\end{lem}
\begin{proof} \ Take a right $\omega$-approximation $t: W\longrightarrow C$ of $C$. Since $\beta\alpha\in {\rm Weq}_\omega$, it follows from
Lemma \ref{necessityofWeq} that $(\beta\alpha, t): A\oplus W \longrightarrow C$ is in ${\rm TFib}_\omega$, i.e., $(\beta\alpha, t)$ is a deflation and $\Ker (\beta\alpha, t)\in \Y.$
Since $(\beta\alpha, t)=(\beta, t)\left(\begin{smallmatrix} \alpha & 0\\ 0 & 1\end{smallmatrix}\right)$,
\ $(\beta, t): B\oplus W \longrightarrow C$ is a deflation.

\vskip5pt

Since  $\left(\begin{smallmatrix}\alpha&0\\0&1 \end{smallmatrix}\right): A\oplus W\longrightarrow B\oplus W$ is a deflation with
$\Ker \left(\begin{smallmatrix}\alpha&0\\0&1 \end{smallmatrix}\right) = \Ker \alpha$ and
$(\beta\alpha, t) = (\beta, t)\left(\begin{smallmatrix}\alpha&0\\0&1 \end{smallmatrix}\right)$, it follows from Lemma \ref{twodeflations}(1) that
there is an admissible exact sequence
$$0\longrightarrow \Ker \left(\begin{smallmatrix}\alpha &0\\0&1 \end{smallmatrix}\right) \longrightarrow \Ker (\beta\alpha, t) \longrightarrow
\Ker (\beta, t)\longrightarrow 0.$$
Note that $\Ker \left(\begin{smallmatrix}\alpha&0\\0&1 \end{smallmatrix}\right) = \Ker \alpha\in\Y$ and $\Ker (\beta\alpha, t)\in \mathcal{Y}$. Since $(\X, \Y)$ is a hereditary cotorsion pair,
it follows that
$\Ker (\beta, t)\in \Y$. Thus $(\beta, t)\in {\rm TFib}_\omega$, and hence by definition $\beta\in {\rm Weq}_\omega$.
\end{proof}

\begin{lem} \label{232} \ Let $\alpha: A\longrightarrow B$ and $\beta: B\longrightarrow C$ be morphisms in $\A$ such that $\alpha$ and $\beta\alpha$ are in ${\rm Weq}_\omega$. Then $\beta \in {\rm Weq}_\omega$.
\end{lem}
\begin{proof}  Since $\alpha\in {\rm Weq}_\omega$, there is morphism $(\alpha, t): A\oplus W\longrightarrow B$
with $W\in \omega$ such that $(\alpha, t)\in {\rm TFib}_\omega$.
To prove $\beta \in {\rm Weq}_\omega$, by the commutative diagram
\[
\xymatrix@R=0.5cm{A\oplus W\ar[dr]_-{(\alpha, t)}\ar[rr]^-{(\beta\alpha, \beta t)} && C\\ & B\ar[ur]_-{\beta}}\]
and by Lemma \ref{TFibWeq}, it suffices to prove $(\beta\alpha, \beta t)\in {\rm Weq}_\omega$.

\vskip5pt

Take  a right $\omega$-approximation $t': W'\longrightarrow C$ of $C$. Since $\beta\alpha\in {\rm Weq}_\omega$,
it follows from Lemma \ref{necessityofWeq} that $(\beta\alpha, t'): A\oplus W'\longrightarrow C$ is in ${\rm TFib}_\omega$, i.e., $(\beta\alpha, t')$ is a deflation and $\Ker (\beta\alpha, t')\in \Y.$
Since  $\left(\begin{smallmatrix}1&0\\0&0\\0&1 \end{smallmatrix}\right): A\oplus W'\longrightarrow A\oplus W\oplus W'$
is a splitting inflation with cokernel $W$ and since $(\beta\alpha, t') = (\beta\alpha, \beta t, t')\left(\begin{smallmatrix}1&0\\0&0\\0&1 \end{smallmatrix}\right)$ is a deflation,
it follows from Lemma \ref{twodeflations}(2)  that
there is an admissible exact sequence
$$0\longrightarrow \Ker (\beta\alpha, t') \longrightarrow \Ker (\beta\alpha, \beta t, t') \longrightarrow
\Coker \left(\begin{smallmatrix} 1 & 0\\0 & 0\\ 0 & 1 \end{smallmatrix}\right)\longrightarrow 0.$$
See the diagram below.
\[
\xymatrix@R=0.5cm{
& 0\ar[d] & 0\ar[d] & &\\
0 \ar[r]& \Ker (\beta\alpha, t')\ar[r]\ar[d] &A\oplus W'\ar[d]_-{\left(\begin{smallmatrix}1&0\\0&0\\0&1 \end{smallmatrix}\right)}\ar[rr]^-{(\beta\alpha, t')}&& C\ar[r]\ar@{=}[d]& 0\\
0 \ar[r]& \Ker (\beta\alpha,\beta t, t')\ar[r]\ar[d] & A\oplus W\oplus W'\ar[d]\ar[rr]^-{(\beta\alpha, \beta t, t')} && C\ar[r]& 0\\
& W\ar@{=}[r]\ar[d] & W\ar[d] & &\\
& 0 & 0 &&  \\
}
\]
Since $\Ker (\beta\alpha, t')\in\Y$ and $\Coker \left(\begin{smallmatrix} 1 & 0\\0 & 0\\ 0 & 1 \end{smallmatrix}\right) = W\in \mathcal{Y}$. It follows that $\Ker (\beta\alpha, \beta t, t')\in \Y$, and hence $(\beta\alpha, \beta t, t')\in {\rm TFib}_\omega.$
By the commutative diagram
\[
\xymatrix@R=0.5cm{A\oplus W\ar[dr]_-{\left(\begin{smallmatrix}1&0\\0&1\\0&0 \end{smallmatrix}\right)}\ar[rr]^-{(\beta\alpha, \beta t)} && C\\ & A\oplus W\oplus W'\ar[ur]_-{(\beta\alpha, \beta t, t')}}\]
we see that $(\beta\alpha, \beta t)\in {\rm Weq}_\omega$. This completes the proof.
\end{proof}

\begin{lem}\label{CofTFib} \ Let $\alpha=pi\in {\rm Weq}_{\omega}$ with $i\in {\rm CoFib}_{\omega}$ and $p\in {\rm TFib}_{\omega}$. Then $i\in {\rm TCoFib}_{\omega}$.
\end{lem}
\begin{proof} \ We first show that $i$ splits. Since $i\in {\rm CoFib}_{\omega}$, $i$ is an inflation with $\Coker i\in \mathcal X$.
Since $\alpha\in {\rm Weq}_{\omega}$, by definition there is a deflation $(\alpha, t): A\oplus W\longrightarrow B$
with $W\in \omega$ and $\Ker (\alpha, t)\in \Y$. Consider the commutative diagram with admissible exact rows
\[
\xymatrix@R=0.5cm{
& 0\ar[r]&A\ar[r]^-{i}\ar[d]_-{\left(\begin{smallmatrix}1 \\ 0\end{smallmatrix}\right)}&C\ar[r]\ar[d]^-{p}&\Coker i\ar[r]&0\\
0\ar[r]& \Ker (\alpha, t)\ar[r]&A\oplus W\ar[r]^-{(\alpha,t)}&B\ar[r]&0&
}
\]
By the Extension-Lifting Lemma \ref{extlifting}, there is a lifting $\begin{pmatrix}\sigma_1,\sigma_2 \end{pmatrix}: C\longrightarrow A\oplus W$ such that $\sigma_1 i=1_A$. So $i$ splits.

\vskip5pt

Thus, one can write $\alpha=pi$ as
\[
\xymatrix@R=0.5cm{A\ar[dr]_-{i=\left(\begin{smallmatrix}1\\ 0  \end{smallmatrix}\right)}\ar[rr]^-\alpha && B\\ & A\oplus X\ar[ur]_-{p=(\alpha, \alpha')}}\]
with $X = \Coker i\in \mathcal X$ and $p=(\alpha, \alpha')\in {\rm TFib}_\omega$.
It remains to prove that $X\in \mathcal{Y}$.

\vskip5pt

Since $\left (\begin{smallmatrix}1&0\\0&0\\0&1 \end{smallmatrix}\right): A\oplus X\longrightarrow A\oplus W\oplus X$ is an inflation and
$p= (\alpha, \alpha') = (\alpha, t,\alpha')\left (\begin{smallmatrix}1&0\\0&0\\0&1 \end{smallmatrix}\right)$ is a deflation, where
$(\alpha, t, \alpha'):  A\oplus W\oplus X \longrightarrow B$,
it follows from Lemma \ref{twodeflations}(2) that there is an admissible exact sequence
$$0\longrightarrow \Ker (\alpha, \alpha') \longrightarrow \Ker (\alpha, t, \alpha') \longrightarrow
\Coker \left (\begin{smallmatrix}1&0\\0&0\\0&1 \end{smallmatrix}\right)\longrightarrow 0.$$
Since $\Ker (\alpha, \alpha')\in \mathcal{Y}$ and $\Coker \left (\begin{smallmatrix}1&0\\0&0\\0&1 \end{smallmatrix}\right)=W\in \mathcal{Y}$, it follows that $\Ker (\alpha, t ,\alpha')\in \Y$.

\vskip5pt

Since $(\alpha, t, \alpha')$ is a deflation, there is an admissible exact sequence
\[
\xymatrix{
0 \ar[r]& \Ker (\alpha, t, \alpha')\ar[r]^-{\left(\begin{smallmatrix}k_1\\k_2\\k_3\end{smallmatrix}\right)} & A\oplus W\oplus X\ar[rr]^-{(\alpha, t, \alpha')} && B\ar[r]& 0.
}
\]
Consider the commutative square
\[
\xymatrix@R=0.6cm{
\Ker (\alpha, t, \alpha')\ar[r]^-{-k_3}\ar[d]_{\left(\begin{smallmatrix}k_1 \\ k_2\end{smallmatrix}\right)}& X\ar[d]^-{\alpha'}\\
A\oplus W\ar[r]^-{(\alpha,t)}&B}
\]
with deflation $(\alpha, t)$, by the equivalence of (ii) and (iv) in Lemma \ref{pp}(1), there is the following commutative diagram with admissible exact rows
\[
\xymatrix@R=0.5cm{
0\ar[r]& \Ker (\alpha,t)\ar[r]\ar@{=}[d]& \Ker (\alpha, t, \alpha')\ar[r]^-{-k_3}\ar[d]_{\left(\begin{smallmatrix}k_1 \\ k_2\end{smallmatrix}\right)}& X\ar[r]\ar[d]^-{\alpha'}
& 0\\
0\ar[r] & \Ker (\alpha, t)\ar[r]& A\oplus W\ar[r]^-{(\alpha,t)}&B\ar[r]& 0}
\]
In the admissible exact sequence  of the first row,  $\Ker (\alpha,t)\in\Y$ and $\Ker (\alpha,\alpha',t)\in \mathcal{Y}$. By the assumption
$(\mathcal{X}, \Y)$ is a hereditary cotorsion pair,  it follows that  $X\in\Y$. This completes the proof. \end{proof}

\begin{lem} \label{233} \ \ Let $\alpha: A\longrightarrow B$ and $\beta: B\longrightarrow C$ be morphisms in $\A$ with  $\beta\in {\rm Weq}_\omega$ and $\beta\alpha\in {\rm Weq}_\omega$. Then $\alpha \in {\rm Weq}_\omega$.
\end{lem}
\begin{proof} \ Since $\beta\in {\rm Weq}_\omega$, there is a morphism $(\beta, t): B\oplus W\longrightarrow C$ which is in ${\rm TFib}_\omega$
with $W\in \omega$.
By the factorization axiom, which has been already proved,  one can decompose $\left(\begin{smallmatrix} \alpha \\ 0\end{smallmatrix}\right): A\longrightarrow B\oplus W$
as $\left(\begin{smallmatrix} \alpha \\ 0\end{smallmatrix}\right)=\left(\begin{smallmatrix} p_1 \\ p_2\end{smallmatrix}\right) i$ with  $i\in {\rm CoFib}_{\omega}$ and $\left(\begin{smallmatrix} p_1 \\ p_2 \end{smallmatrix}\right)\in {\rm TFib}_{\omega}$. By Lemma \ref{TFib}, $(\beta, t)\left(\begin{smallmatrix} p_1 \\ p_2 \end{smallmatrix}\right)\in {\rm TFib}_{\omega}.$ Write
$$\beta\alpha = (\beta, t)\left(\begin{smallmatrix} \alpha\\ 0 \end{smallmatrix}\right) = (\beta, t)\left(\begin{smallmatrix} p_1 \\ p_2 \end{smallmatrix}\right) i$$ where $\beta\alpha\in {\rm Weq}_{\omega},$ $i\in {\rm CoFib}_{\omega}$ and $(\beta, t)\left(\begin{smallmatrix} p_1 \\ p_2 \end{smallmatrix}\right)\in {\rm TFib}_{\omega}.$
By Lemma \ref{CofTFib} one has $i\in {\rm TCofib}_{\omega}$. It follows that $\left(\begin{smallmatrix} \alpha \\ 0\end{smallmatrix}\right)=\left(\begin{smallmatrix} p_1 \\ p_2\end{smallmatrix}\right) i\in {\rm Weq}_{\omega}$. By Lemma \ref{231} and $\alpha=(1,0)\left(\begin{smallmatrix} \alpha \\ 0\end{smallmatrix}\right)$ one sees that $\alpha\in {\rm Weq}_{\omega}$, since  $(1,0): B\oplus W \longrightarrow B$ is in ${\rm TFib}_{\omega} \subseteq {\rm Weq}_{\omega}$. \end{proof}

\vskip5pt

{\bf Proof of the two out of three axiom}. \ Now, the two out of three axiom, i.e., Lemma \ref{23},  follows from Lemma \ref{231}, Lemma \ref{232} and Lemma \ref{233}. \hfill $\square$

\subsection{Retract axiom}

The aim of this subsection is to prove that $\CoFib_{\omega}, \Fib_{\omega}, \Weq_{\omega}$ are closed under retract.
Suppose that $g: A' \longrightarrow B'$ is a retract of $f: A \longrightarrow B$, i.e., one has a commutative diagram of morphisms
$$\xymatrix@R=0.5cm{
    A'\ar[r]^{\varphi_1}\ar[d]_g & A\ar[r]^{\psi_1}\ar[d]_{f} & A'\ar[d]_{g} \\
    B'\ar[r]^{\varphi_2} & B\ar[r]^{\psi_2} & B'
}$$
with  $\psi_1 \varphi_1 =\Id_{A'}$ and $\psi_2 \varphi_2 = \Id_{B'}$.

\vskip5pt

{\bf Step 1. \ $\CoFib_{\omega}$ is closed under retract. }

\vskip5pt

Let $f \in \CoFib_{\omega}$, i.e., $f$ is an inflation with cokernel in $\mathcal{X}.$ Then one has a commutative diagram
$$\xymatrix@R=0.5cm{
    A'\ar@<.5ex>[r]^{\varphi_1}\ar[d]_g & A\ar[d]^f\ar@<.5ex>[l]^{\psi_1}  \\
    B'\ar@<.5ex>[r]^{\varphi_2}\ar[d]_{c_g} & B\ar[d]^{c_f}\ar@<.5ex>[l]^{\psi_2}  \\
    \Coker g\ar@<.5ex>[r]^{\widetilde{\varphi_2}} & \Coker f. \ar@<.5ex>[l]^{\widetilde{\psi_2}}
}$$
Since $\varphi_2 g = f \varphi_1$ is an inflation, $g$ is an inflation. Then $\widetilde{\psi_2} \widetilde{\varphi_2}c_g = \widetilde{\psi_2}c_f \varphi_2 = c_g \psi_2 \varphi_2 = c_g$. Since $c_g$ is a deflation,  $\widetilde{\psi_2} \widetilde{\varphi_2} = \Id_{\Coker g}$.
Thus $\Coker g$ is a direct summand of $\Coker f$, inducing that $\Coker g\in \X$. By definition $g \in \CoFib_\omega$.

\vskip5pt

{\bf Step 2. \ $\Fib_{\omega}$ is closed under retract.}

\vskip5pt

 Let $f \in \Fib_{\omega}$. For any $W\in \omega$ and any morphism $t: W\longrightarrow B'$, since $f$ is $\omega$-epic, there is a morphism $s$ such that $f s = \varphi_2 t$. See the following diagram
$$\xymatrix{
     W\ar[dr]^t\ar@{-->}@/^20pt/[rr]^-{s}& A'\ar@<.5ex>[r]^{\varphi_1}\ar[d]_g & A\ar[d]^f\ar@<.5ex>[l]^{\psi_1}  \\
    &B'\ar@<.5ex>[r]^{\varphi_2} & B\ar@<.5ex>[l]^{\psi_2}
}$$
Then  $g\psi_1 s = \psi_2 f s = \psi_2 \varphi_2 t = t$. Thus $g$ is also $\omega$-epic. By definition $g\in \Fib_{\omega}$.

\vskip5pt

{\bf Step 3. \ $\Weq_{\omega}$ is closed under retract.} The proof below is also different from
the one for abelian categories ([BR, VIII, Theorem 4.2]) which involves left triangulated categories.

\vskip5pt

Let $f \in \Weq_{\omega}$. Then there is  a deflation $(f, \alpha): A \oplus W \longrightarrow B$  with $W\in \omega$ and $\Ker (f,\alpha) \in \mathcal{Y}$.
Since $(g, \psi_2 \alpha)\left(\begin{smallmatrix} \psi_1 & 0 \\ 0 & 1 \end{smallmatrix}\right)=\psi_2 (f,\alpha)$ is a deflation,  $(g,\psi_2 \alpha)$ is a deflation, say with kernel $\left(\begin{smallmatrix} k_1 \\ k_2 \end{smallmatrix}\right): K \longrightarrow A' \oplus W$.
To show that $g \in \Weq_{\omega}$, it suffices to show that $K \in \mathcal{Y}$. Since $g$ is a retract of $f$, one has a commutative diagram with admissible exact
rows:
$$\xymatrix@R=0.5cm{
    0\ar[r] & A'\ar@<0.5ex>[r]^{\varphi_1}\ar[d]_g & A\ar@<0.5ex>[r]^{\partial_1}\ar[d]^f\ar@{-->}@<0.5ex>[l]^{\psi_1} & A''\ar[r]\ar[d]\ar@{-->}@<0.5ex>[l]^{\delta_1} & 0 \\
    0\ar[r] & B'\ar@<0.5ex>[r]^{\varphi_2} & B\ar@{-->}@<0.5ex>[l]^{\psi_2}\ar@<0.5ex>[r]^{\partial_2} & B''\ar[r]\ar@{-->}@<0.5ex>[l]^{\delta_2} & 0
}$$
where $\varphi_2 \psi_2 + \delta_2 \partial_2 = \Id_B$. Since $(\X, \Y)$ is a complete cotorsion pair, there is an admissible exact sequence
$$0 \longrightarrow K \overset{i}{\longrightarrow} Y \overset{d}{\longrightarrow} X \longrightarrow 0$$
with $Y \in \mathcal{Y}$ and $X \in \mathcal{X}$. Since $W \in \omega \subseteq \mathcal{Y}$, there exists a morphism $s: Y \longrightarrow W$ such that $k_2 = si$. Then one has the following diagram
$$\xymatrix@R=0.6cm{
&0\ar[r]&K\ar[r]^-i\ar[d]_-{\left(\begin{smallmatrix} \varphi_1 k_1 \\ k_2 \end{smallmatrix}\right)}&Y\ar[r]^-d\ar[d]^-{\delta_2\partial_2\alpha s}\ar@{-->}[dl]_-{\left(\begin{smallmatrix} m \\ n \end{smallmatrix}\right)}& X\ar[r]&0\\
0\ar[r]&\Ker (f, \alpha)\ar[r]&A \oplus W \ar[r]_-{(f,\alpha)}&B\ar[r]& 0&
}$$

Since
\begin{equation*}
    \begin{aligned}
 (f,\alpha) \left(\begin{smallmatrix} \varphi_1 k_1 \\ k_2 \end{smallmatrix}\right) - \delta_2\partial_2\alpha si & = f \varphi_1 k_1 + \alpha k_2 - \delta_2\partial_2\alpha k_2 \\
         & = \varphi_2 g k_1 + \varphi_2 \psi_2 \alpha k_2 \\
         & = \varphi_2 \ (g,\psi_2 \alpha)\left(\begin{smallmatrix} k_1 \\ k_2 \end{smallmatrix}\right) = 0,
    \end{aligned}
\end{equation*}
it follows from the Extension-Lifting Lemma \ref{extlifting} that there is a morphism $\left(\begin{smallmatrix} m  \\ n \end{smallmatrix}\right): Y \longrightarrow A \oplus W$ such that
$mi = \varphi_1 k_1, \ \ ni = k_2, \ \ fm + \alpha n  = \delta_2\partial_2\alpha s$.
Consider the diagram
$$\xymatrix{& & Y\ar[d]^-{\left(\begin{smallmatrix} \psi_1m \\ n \end{smallmatrix}\right)} \\
0\ar[r] & K\ar[r]^-{\left(\begin{smallmatrix} k_1 \\ k_2 \end{smallmatrix}\right)} & A' \oplus W \ar[rr]^{(g, \psi_2 \alpha)} && B'\ar[r] & 0.}$$
Since
\begin{align*}
    (g,\psi_2\alpha)\left(\begin{smallmatrix} \psi_1 m  \\ n \end{smallmatrix}\right) & = g\psi_1 m +\psi_2\alpha n  = \psi_2 fm +\psi_2 \alpha n \\
    & = \psi_2 \delta_2 \partial_2 \alpha s =0,
\end{align*}
there exists a morphism $t: Y \longrightarrow K$ such that $\left(\begin{smallmatrix} \psi_1 m  \\ n \end{smallmatrix}\right) = \left(\begin{smallmatrix} k_1 \\ k_2 \end{smallmatrix}\right) t$. Then
$$\left(\begin{smallmatrix} k_1 \\ k_2 \end{smallmatrix}\right) ti = \left(\begin{smallmatrix} \psi_1 mi \\ ni \end{smallmatrix}\right) = \left(\begin{smallmatrix} \psi_1\varphi_1k_1 \\ k_2\end{smallmatrix}\right) = \left(\begin{smallmatrix} k_1 \\ k_2 \end{smallmatrix}\right).$$ Since $\left(\begin{smallmatrix} k_1 \\ k_2 \end{smallmatrix}\right)$ is an inflation, $ti = \Id_K$. Thus $K$ is a direct summand of $Y$, and hence $K \in \mathcal{Y}$. \hfill $\square$

\subsection{Lifting axiom} This subsection is to prove the lifting axiom. Let
\[
\xymatrix@R=0.5cm{
A\ar[r]^-{f}\ar[d]_-{i} & C\ar[d]^-{p}\\
B\ar[r]^-{g} & D
}
\]
be a commutative square with $i\in \CoFib_\omega$ and $p\in \Fib_\omega$.

\vskip5pt

{\bf Case 1. \ Suppose that $p\in {\rm Fib}_{\omega}\cap \Weq_{\omega}$ }. By Lemma \ref{tcofibandtfib},  $p\in \TFib_\omega$, i.e., $p$ is a deflation with $\Ker p\in \Y$.  Then the lifting indeed exists,  directly by the Extension-Lifting Lemma \ref{extlifting}.

\vskip5pt

{\bf Case 2. \ Suppose that $i\in {\rm CoFib}_{\omega}\cap \Weq_{\omega}$}.  By Lemma \ref{tcofibandtfib},  $i\in\TCoFib_\omega$, i.e., $i$ is a splitting monomorphism with $\Coker i \in \omega$. Thus we can rewrite the commutative square as
\[
\xymatrix{
A\ar[r]^-{f}\ar[d]_-{\left(\begin{smallmatrix}1\\0\end{smallmatrix}\right)} & C\ar[d]^-{p}\\
A\oplus W\ar[r]^-{(pf,g')} & D
}
\]
with $W=\Coker i\in \omega$. Since $p$ is $\omega$-epic, there is a morphism $s:W\longrightarrow C$ such that $g'=ps$. Then there is a lifting $(f, s): A\oplus W\longrightarrow C$, which completes the proof.
\hfill $\square$

\vskip5pt

\subsection{Proof of Theorem \ref{ifpart}} \ Up to now we have proved that $({\rm CoFib}_{\omega}, \ {\rm Fib}_{\omega}, \ {\rm Weq}_{\omega})$
is a model structure on $\A$. By Lemma \ref{tcofibandtfib}, the class $\TCoFib_\omega$ of trivial cofibrations is precisely
the class of splitting monomorphisms with cokernel in $\omega$, and the class $\TFib_\omega$
of trivial fibrations is precisely the class of deflations with kernel in $\mathcal Y.$  Thus, Theorem \ref{ifpart} is proved.

\subsection{\bf When is the $\omega$-model structure exact?}

It is natural to ask when the model structure $(\CoFib_{\omega}, \Fib_{\omega}, \Weq_{\omega})$  is exact.
Since by definition $\CoFib_{\omega} = \{\mbox{inflation} \ f \ | \ \Coker f\in\mathcal X\}$ and $\Fib_{\omega} = \{\mbox{morphism} \ f \ | \ f \ \mbox{is} \ \omega\mbox{-epic}\}$,
one easily knows that the classes of cofibrant objects and of fibrant objects of model structure $(\CoFib_{\omega}, \Fib_{\omega}, \Weq_{\omega})$  are respectively $\mathcal X$ and $\mathcal A$.
Recall that a model structure on an exact category is exact ([G, 3.1]), if
cofibrations are precisely inflations with cofibrant cokernel, and fibrations are precisely
deflations with fibrant kernel. Thus, the model structure $(\CoFib_{\omega}, \Fib_{\omega}, \Weq_{\omega})$  is exact if and only if
$$\{\mbox{morphism} \ f  \mid  f \ \mbox{is} \ \omega\mbox{-epic}\} = \{\mbox{morphism} \ f \ | \ f \ \text{is a deflation}\}.$$

Recall that by definition an object $P$ is {\it projective} in exact category $\mathcal A$, if for any deflation $d$, the map $\Hom_\mathcal A(P, d)$ is surjective; and that
$\mathcal A$ {\it has enough projective objects}, if for any object $X\in\mathcal A$ there is a deflation $P\longrightarrow X$ with  $P$ a projective object.

\begin{prop}\label{thesame} \ Let $\mathcal A$ be a weakly idempotent complete exact category,
\ $(\mathcal{X}, \mathcal{Y})$ a hereditary complete cotorsion pair with $\omega=\mathcal{X}\cap \mathcal{Y}$ contravariantly finite.
Then the model structure $(\CoFib_{\omega}, \Fib_{\omega}, \Weq_{\omega})$ is exact if and only if  $\mathcal A$ has enough projective objects
and $\omega = \mathcal P$, the class of projective objects of $\mathcal A$.
\end{prop}

\begin{proof} \ If $\mathcal A$ has enough projective objects and $\omega = \mathcal P$, and $f: A\longrightarrow B$ is $\omega$-epic, taking a deflation $g:P\longrightarrow B$ with  $P$ a projective object, then
$g = fh$ for some $h: P\longrightarrow A$. Since $\mathcal{A}$ is weakly idempotent complete, $f$ is a deflation. So
$\{f  \mid  f \ \mbox{is} \ \omega\mbox{-epic}\} = \{f \ \text{is a deflation}\}$.

\vskip5pt

Conversely, assume that the model structure $(\CoFib_{\omega}, \Fib_{\omega}, \Weq_{\omega})$  is exact. By the Hovey correspondence
$(\mathcal X, \ \A, \ \Y)$ is a Hovey triple, and hence $(\omega, \mathcal{A})$ is a complete cotorsion pair, so $\omega = \ ^\perp\mathcal A = \mathcal{P}$ and $\mathcal A$ has enough projective objects.
\end{proof}

\subsection{A class of non exact model structures in exact categories which are not abelian}

\begin{exm} \label{newmodel} \  Let $\Lambda$ be an Artin algebra, $\Lambda$-mod the category of finitely generated left $\Lambda$-modules.
For a module $M$, let ${\rm add} M$ be the class of modules which are summands of finite direct sums of copies of $M$,
and $\widetilde{{\rm add} M}$ the class of modules $X$ with an ${\rm add} M$-coresolution, that is, there is an exact sequence
\[
\xymatrix{
0\ar[r]& X\ar[r]& M^0\ar[r]& \ldots\ar[r]& M^s\ar[r]& 0
}
\]
with each $M^i\in {\rm add} M$. Define $\widehat{{\rm add} M}$ dually, i.e., the class of modules with an ${\rm add} M$-resolution.

\vskip5pt

Let $T$ be a tilting module, i.e., ${\rm proj.dim.} T <\infty$, \ $\Ext_\Lambda^i(T, T)=0$ for $i\ge 1$, and  $\Lambda\in \widetilde{{\rm add}T}$.
Following [AR], put $\mathcal P^{<\infty}$ to be $\widehat{{\rm add} \Lambda}$,  the class of modules of finite projective dimension.
Then $\mathcal P^{<\infty}$ is a weakly idempotent complete exact category; and $\mathcal P^{<\infty}$ is not an abelian category if and only if
the global dimension of $\Lambda$ is infinite. (In fact, if ${\rm proj.dim}M = \infty$, taking a projective presentation $Q\stackrel f{\longrightarrow} P \longrightarrow M\longrightarrow 0$, then
the morphism $f: Q\longrightarrow P$  has no cokernel in $\mathcal P^{<\infty}$.)

\vskip5pt

Let $T$ be a tilting module. Then $T$ is a tilting object in exact category $\mathcal P^{<\infty}$, in the sense of Krause [Kr, p. 215], i.e.,
$\Ext_\Lambda^i(T, T)=0$ for $i\ge 1$,  and ${\rm Thick}(T)$, the smallest thick subcategory of $\mathcal A$ containing $T$, is just $\mathcal P^{<\infty}$.
By \cite[7.2.1]{Kr},  $(\widetilde{{\rm add} T}, \ \widehat{{\rm add} T})$ is a hereditary complete  cotorsion pair
in exact category $\mathcal P^{<\infty}$,  with $\omega:=\widetilde{{\rm add} T}\cap \widehat{{\rm add} T}={\rm add} T$ contravariantly finite   in $\mathcal P^{<\infty}$.

\vskip5pt

If $T$ is not a projective module, then by Proposition \ref{thesame}, the model structure on exact category $\mathcal P^{<\infty}$ induced by the hereditary complete  cotorsion pair
$(\widetilde{{\rm add} T}, \ \widehat{{\rm add} T})$ is not exact.
\end{exm}
\begin{exm} \label{generalexm} \ More general, let $\mathcal A$ be an abelian category, $\mathcal E$  an orthogonal full subcategory of $\mathcal A$, i.e., $\Ext_\mathcal A^i(X,Y)=0$ for any $X,Y\in \mathcal E$ and $i\geq 1$.
Then ${\rm Thick}(\mathcal E)$ is a weakly idempotent complete exact category. By \cite[7.1.10]{Kr}, $(\widetilde{\mathcal E}, \ \widehat{\mathcal E})$ is a hereditary complete cotorsion pair in ${\rm Thick}(\mathcal E)$ with core $\mathcal E = \widetilde{\mathcal E}\cap\widehat{\mathcal E}$. If moreover $\mathcal E$ is contravariantly finite in $\mathcal A$, then so is $\mathcal E$ in ${\rm Thick}(\mathcal E)$, and hence $(\CoFib_{\mathcal E}, \Fib_{\mathcal E}, \Weq_{\mathcal E})$ is a model structure  in ${\rm Thick}(\mathcal E)$.
\end{exm}

\subsection{A non-hereditary complete cotorsion pair  with core contravariantly finite} \ We claim that the condition $(\X, \Y)$ is hereditary in Theorem \ref{mainthm} is essential.
The following example shows that there does exist a complete cotorsion pair $(\X, \Y)$ with $\omega=\X\cap \Y$ contravariantly finite such that $(\X, \Y)$ is not hereditary, and hence
 $({\rm CoFib}_{\omega}, \ {\rm Fib}_{\omega}, \ {\rm Weq}_{\omega})$ is not a model structure, by Proposition \ref{resolvingcoresolving}.

\begin{exm} \label{nothereditary} \ Let $k$ be a field, $Q$ the quiver \ $\xymatrix{3\ar[r]^-{\beta}& 2\ar[r]^-{\alpha}&1}$  and $A =kQ/\langle \alpha\beta\rangle$.
The  Auslander-Reiten quiver of $A$ is
$$\xymatrix@C=1.2pc @ R = 0.3pc{& P(2)\ar[dr]&& \\ S(1)\ar[ur]\ar@{.}[rr] & & S(2)\ar[dr]\ar@{.}[rr] & & S(3)
\\ & &&P(3)\ar[ur]}$$
Consider the full subcategory  $\mathcal{C}:=\text{add} (_AA\oplus S(3))$ of $A\text{-mod}$,  the category of finitely generated left $A$-modules.
From the Auslander-Reiten quiver of $A$ one easily sees that $(\mathcal{C}, \mathcal{C})$ is a complete cotorsion pair in $A\text{-mod}$. For example,
if $X$ is an indecomposable $A$-module such that $\Ext^1_A(X, \mathcal C) = 0$,  then $X\ne S(2)$, thus $X\in \text{add} (_AA\oplus S(3)) = \mathcal C.$
Also, by definition the cotorsion pair $(\mathcal{C}, \mathcal{C})$ is complete, which essentially follows from the exact sequences
$0 \longrightarrow S(1) \longrightarrow P(2) \longrightarrow S(2) \longrightarrow 0$ and $0 \longrightarrow S(2) \longrightarrow P(3) \longrightarrow S(3) \longrightarrow 0$.

\vskip5pt

For any module $M$, it is well-known that $\text{add}(M)$ is contravariantly finite in $A\text{-mod}$. In fact,
let $M_1, \cdots, M_n$ be the pairwise non-isomorphic indecomposable direct summands of $M$, and for any module $X$, let $f_{i1}, \cdots, f_{it_i}$ be a
$k$-bases of $\Hom_A(M_i, X)$, \ $1\le i \le n.$ Then
$$\xymatrix{M_1^{t_1}\oplus \cdots \oplus M_n^{t_n} \ar[rrrr]^-{(f_{11}, \cdots, f_{1t_1}, \cdots, f_{n1}, \cdots, f_{nt_n})} &&&& X}$$
is a right $\text{add}(M)$-approximation of $X$. Thus,  $\omega: = \mathcal{C}\cap\mathcal{C} = \mathcal{C}$ is contravariantly finite in $A\text{-mod}$.

\vskip5pt
Note that  the cotorsion pair $(\mathcal{C}, \mathcal C)$ is not hereditary, since there is an exact sequence
\[
\xymatrix{
0\ar[r]& S(2)\ar[r]& P(3)\ar[r]&S(3)\ar[r]&0
}
\]
or, since $\Ext_A^2(S(3), S(1))\ne 0.$    \ Thus by Proposition \ref{resolvingcoresolving}, $({\rm CoFib}_{\omega}, \ {\rm Fib}_{\omega}, \ {\rm Weq}_{\omega})$ is not a model structure on $A\text{-mod}$.
\end{exm}

\section{\bf Hereditary complete cotorsion pair arising from a model structure}

The aim of this section is to prove the ``only if" part of Theorem \ref{mainthm}, namely

\begin{thm}\label{onlyif} \ Let $\mathcal A$ be a weakly idempotent complete exact category,
$\mathcal{X}$ and $\mathcal{Y}$ additive full subcategories of $\mathcal{A}$ which are closed under direct summands and isomorphisms, and $\omega=\mathcal{X}\cap \mathcal{Y}$.
If  $({\rm CoFib}_{\omega}, \ {\rm Fib}_{\omega}, \ {\rm Weq}_{\omega})$ is a  model structure,  then  $(\mathcal{X},\mathcal{Y})$ is a hereditary complete cotorsion pair in $\mathcal{A}$, and $\omega$ is contravariantly finite in $\mathcal A$;
and the class $\mathcal C_\omega$ of cofibrant objects is $\mathcal X,$ the class $\mathcal F_\omega$ of fibrant objects is $\mathcal A,$
the class $\mathcal W_\omega$ of trivial objects is $\mathcal Y;$ and the homotopy category
${\rm Ho}(\mathcal A)$ is $\mathcal X/\omega$.
\end{thm}

\subsection{Complete cotorsion pairs} Let $(\CoFib ,\  \Fib ,\  \Weq )$ be a model structure on an arbitrary category $\mathcal{A}$ with zero object.  Put
\begin{align*} & \mathcal C: = \{\mbox{cofibrant objects}\}, \ \ \ \mathcal F: = \{\mbox{fibrant objects}\},
\ \ \ \mathcal W: = \{\mbox{trivial objects}\} \\
& {\rm T}\mathcal C: = \{\mbox{trivially cofibrant objects}\} , \ \ \ {\rm T}\mathcal F: = \{\mbox{trivially fibrant objects}\}.\end{align*}

The proof of the following two lemmas is the same as in additive categories.

\begin{lem} \label{factorthrough} \ {\rm ([BR, VIII, 1.1)} \ Let $(\CoFib ,\  \Fib ,\  \Weq )$ be a model structure on an arbitrary category $\mathcal{A}$ with zero object. Then

\vskip5pt

$(1)$ \ If $p: B \longrightarrow C$ is a trivial fibration $($respectively, a fibration$)$,
then any morphism $\gamma: X \longrightarrow C$ factors through $p$, where $X \in \mathcal{C}$ $($respectively, $X\in {\rm T}\mathcal{C})$.

\vskip5pt

$(2)$ \ If $i: A \longrightarrow B$ is a trivial cofibration $($respectively, a cofibration$)$, then any morphism $\alpha: A \longrightarrow Y$ factors through $i$,
where $Y \in \mathcal{F}$ $($respectively, $Y\in {\rm T}\mathcal{F})$.

\vskip5pt

$(3)$ \ If $p$ is a fibration $($respectively, a trivial fibration$)$ and $p$ has kernel $F$, then $F \in \mathcal{F}$  \ $($respectively,
$F\in {\rm T}\mathcal{F})$.

\vskip5pt

$(4)$ \ If $i$ is a cofibration $($respectively, a trivial cofibration$)$ and $i$ has cokernel $C$, then $C \in \mathcal{C}$  \ $($respectively, $C\in {\rm T}\mathcal{C})$.

\end{lem}

\begin{lem}\label{contravariantlyfinite} \ {\rm ([BR, VIII, 2.1)} \ Let $(\CoFib ,\  \Fib ,\  \Weq )$ be a model structure on an arbitrary category $\mathcal{A}$ with zero object. Then

\vskip5pt

$(1)$  \ The full subcategory $\mathcal{C}$ is contravariantly finite in $\mathcal{A}$. Furthermore,  for any object    $A$ of $\mathcal{A}$, there exists a right $\mathcal{C}$-approximation  $f_A: C_A \longrightarrow A$ with $f_A \in \TFib;$ and moreover, if $f_A$ admits a kernel, then $\Ker f_A \in {\rm T}\mathcal{F}$.

\vskip5pt

$(2)$  \ The full subcategory $\mathcal{F}$ is covariantly finite in $\mathcal{A}$. Furthermore,  for any object    $A$ of $\mathcal{A}$, there exists
a left $\mathcal{F}$-approximation  $g^A: A \longrightarrow F^A$ with $g^A \in \TCoFib;$ and moreover, if $g^A$ admits a cokernel, then $\Coker g^A \in {\rm T}\mathcal{C}$.

\vskip5pt

$(3)$  \ The full subcategory ${\rm T}\mathcal{C}$ is contravariantly finite in $\mathcal{A}$. Furthermore,  for any object    $A$ of $\mathcal{A}$, there exists a right
${\rm T}\mathcal{C}$-approximation  $\phi_A: X_A \longrightarrow A$ with $\phi_A \in \Fib;$ and moreover, if $\phi_A$ admits a kernel, then $\Ker \phi_A \in \mathcal{F}$.

\vskip5pt

$(4)$  \ The full subcategory ${\rm T}\mathcal{F}$ is covariantly finite in $\mathcal{A}$.
Furthermore, for any object $A$ of $\mathcal{A}$, there exists
a left ${\rm T}\mathcal{F}$-approximation  $\psi^A: A \longrightarrow Y^A$ with $\psi^A \in \CoFib$; and moreover, if $\psi^A$ admits a cokernel, then $\Coker \psi^A \in \mathcal{C}$.
\end{lem}

For abelian categories, the following result is in [BR, VIII, Lemma 3.2], with a slight difference.

\begin{lem}\label{extinclusion} \ {\rm ([BR, VIII, 3.2])} \ Let  $(\CoFib ,\  \Fib ,\  \Weq )$ be a model structure on exact category $\mathcal A$.

\vskip5pt

$(1)$ \ If any inflation with cofibrant cokernel is a cofibration, then $\Ext^1_{\mathcal{A}} (\mathcal{C}, {\rm T}\mathcal{F}) = 0.$

\vskip5pt

$(2)$ \ If any deflation with trivially fibrant kernel is a trivial fibration, then $\Ext^1_{\mathcal{A}} (\mathcal{C}, {\rm T}\mathcal{F}) = 0.$

\vskip5pt

$(3)$  \ If any trivial fibration is a deflation, then $^{\perp}{\rm T}\mathcal{F} \subseteq \mathcal{C}$.

\vskip5pt

$(4)$ \ If any cofibration is an inflation, then $\mathcal{C}^{\perp} \subseteq {\rm T}\mathcal{F}.$

\vskip5pt

$(1')$  \ If any deflation with kernel in $\mathcal{F}$ belongs to $\Fib$, then $\Ext^1_{\mathcal{A}} ({\rm T}\mathcal{C},\mathcal{F}) = 0$.

\vskip5pt

$(2')$ \ If any inflation with trivially cofibrant cokernel is a trivial cofibration, then $\Ext^1_{\mathcal{A}} ({\rm T}\mathcal{C}, \mathcal{F}) = 0.$

\vskip5pt

$(3')$ \ If any trivial cofibration is an inflation, then ${\rm T}\mathcal{C}^{\perp} \subseteq \mathcal{F}.$

\vskip5pt

$(4')$ \ If any fibration is a deflation, then $^{\perp}\mathcal{F} \subseteq {\rm T}\mathcal{C}.$
\end{lem}

\begin{proof} \ By duality it suffices to prove (1) - (4). In fact, the assertion $(1')$ - $(4')$ are only used in the proof of the dual version of
Theorem \ref{mainthm}. The proof of (2) - (4) is the same as in [BR, VIII, 3.2] for abelian categories.
We only justify (1).

\vskip5pt

(1) \ For any admissible exact sequence $0\longrightarrow Y \overset{i}{\longrightarrow} L \overset{d}{\longrightarrow} C\longrightarrow 0$
with $Y \in {\rm T}\mathcal{F}$ and $C \in \mathcal{C}$, by the assumption $i$ is a cofibration.
Thus by Lemma \ref{factorthrough}(2), $\Id_Y : Y \longrightarrow Y$ factors through $i$, i.e., $i$ is a splitting inflation.
\end{proof}

For abelian categories, the following result is in [BR, VIII, 3.4].

\begin{prop}\label{ctpfrommodel} \ Let  $(\CoFib ,\  \Fib ,\  \Weq )$ be a model structure on exact category $\mathcal A$.

\vskip5pt

$(1)$ \ Assume that cofibrations are exactly inflations with cofibrant cokernel and that any trivial fibration is a deflation.
Then $(\mathcal{C}, {\rm T}\mathcal{F})$ is a complete cotorsion pair.

\vskip5pt

$(1')$ \ Assume that fibrations  are exactly  deflations with fibrant kernel and that any trivial cofibration is an inflation.
Then $({\rm T}\mathcal{C}, \mathcal{F})$ is a complete cotorsion pair.
\end{prop}

\begin{proof} \  By duality we only prove (1). By the assumptions and Lemma \ref{extinclusion}(1), (3) and (4),
$(\mathcal{C}, {\rm T}\mathcal{F})$ is a cotorsion pair.

\vskip5pt

By Lemma \ref{contravariantlyfinite}(1), for any object $A \in \mathcal{A}$, there exists a right $\mathcal{C}$-approximation  $f: C \longrightarrow A$ such that $f \in \TFib$. Then by assumption $f$ is a deflation, and hence there is an admissible exact sequence  $0 \longrightarrow Y \longrightarrow C \stackrel f\longrightarrow A \longrightarrow 0$. Then by Lemma \ref{factorthrough}(3) one has $Y \in {\rm T}\mathcal{F}$.

\vskip5pt

Similarly, by Lemma \ref{contravariantlyfinite}(4) and Lemma \ref{factorthrough}(4) one has an admissible exact sequence  $0 \longrightarrow A \longrightarrow Y' \longrightarrow C' \longrightarrow 0$ with $Y' \in {\rm T}\mathcal{F}$ and $C' \in \mathcal{C}$. Thus, the cotorsion pair $(\mathcal{C}, {\rm T}\mathcal{F})$ is complete.
\end{proof}

\subsection{The homotopy category} \ Since we consider the $\omega$-model structure on weakly idempotent complete exact category $\mathcal A$, 
thus  $\mathcal A$ has zero object,  finite coproducts and finite products; and by the axioms of an exact category   
there exist push-outs of two trivial cofibrations and pull-backs of two trivial fibrations (cf. Lemma \ref{tcofibandtfib}).  

\vskip5pt Let $\mathcal{A}_{cf}$ be the full subcategory of $\mathcal{A}$ consisting of all the cofibrant and fibrant objects. Then
${\rm Ho}(\mathcal{A})\cong \pi\mathcal{A}_{cf}.$ See Subsection 2.6.
We will show that $\pi \mathcal{A}_{cf} = \mathcal{X}/ \omega$.
For the model structure $({\rm CoFib}_{\omega}$, \ ${\rm Fib}_{\omega},$ \ ${\rm Weq}_{\omega})$,
$\mathcal{A}_{cf} = \mathcal{X}$. Let $f, g: A \longrightarrow B$ be morphisms with $A, B \in \mathcal{X}$. It suffices to prove the claim:
$f \overset{l}{\sim} g \iff f-g$ factors through $\omega$.
\vskip5pt

If $f \overset{l}{\sim} g$, then one has the commutative diagram
$$\xymatrix@R=0.4cm{
    A \oplus A\ar[rr]^-{(f,g)}\ar[dd]_{(1,1)}\ar[rrdd]^-{(\partial_1,\partial_2)} && B \\ \\
    A && \widetilde{A}\ar[ll]_-{\sigma}\ar[uu]_h
}$$
where $\sigma \in \Weq_{\omega}$. We claim that $\sigma$ can be chosen in $\TFib_\omega$. By definition, there is a deflation $(\sigma, t): \widetilde{A}\oplus W \longrightarrow A$
with $W\in \omega$ and $\Ker (\sigma, t)\in \Y$. Then there is a commutative diagram
 $$\xymatrix@R=0.4cm{
    A \oplus A\ar[rr]^-{(f,g)}\ar[dd]_{(1,1)}\ar[rrdd]^-{\left(\begin{smallmatrix}\partial_1 &\partial_2\\ 0 & 0\end{smallmatrix}\right)} && B \\ \\
    A && \widetilde{A}\oplus W.\ar[ll]_-{(\sigma,t)}\ar[uu]_{(h,0)}
}$$
By definition $(\sigma, t)\in \TFib_{\omega}$. Thus, without loss of generality,  we may assume that $\sigma\in \TFib_{\omega}$. Note that $f-g = h(\partial_1 - \partial_2)$ and $\sigma (\partial_1 -\partial_2) = 0$. It suffices to show that $\partial_1 -\partial_2$ factors through $\omega$. Since $(\mathcal{X}, \mathcal{Y})$ a complete cotorsion pair, one can take an admissible exact sequence $0 \longrightarrow A \overset{i}{\longrightarrow} I \longrightarrow X \longrightarrow 0$ with $I\in\mathcal Y$ and $X\in\mathcal X$.
Then $i \in \CoFib_{\omega}$.  Since $A\in \mathcal{X}$,  $I \in \mathcal{X} \cap \mathcal{Y} = \omega$. By the commutative diagram
$$\xymatrix@R=0.5cm{A\ar[r]^{\partial_1 -\partial_2}\ar[d]_i & \widetilde{A}\ar[d]^{\sigma} \\
    I\ar[r]_0\ar@{-->}[ru] & A
}$$
and the lifting axiom one sees that $\partial_1 -\partial_2$ factors through $\omega$.

\vskip5pt

Conversely,  if $f-g$ factors through $W \in \omega$ by $A \overset{u}{\longrightarrow} W \overset{v}{\longrightarrow} B$, then we have a  diagram
$$\xymatrix@R=0.4cm{
    A \oplus A\ar[rr]^-{(f,g)}\ar[dd]_{(1,1)}\ar[rrdd]^{\left(\begin{smallmatrix} 1 & 1 \\ u & 0 \end{smallmatrix}\right)} && B \\ \\
    A && A \oplus W\ar[ll]^{\sigma = (1,0)}\ar[uu]_{(g,v)}
}$$
where $\sigma \in \TFib_{\omega} \subseteq \Weq_{\omega}$. Thus $f \overset{l}{\sim} g$. This proves the claim, and hence $\Ho (\mathcal{A}) \simeq \mathcal{X}/ \omega$. \hfill $\square$

\subsection{Proof Theorem \ref{onlyif}.} \ Let $\mathcal A$ be a weakly idempotent complete exact category,
$\mathcal{X}$ and $\mathcal{Y}$ full additive subcategories closed under direct summands and isomorphisms, and $\omega:=\mathcal{X}\cap \mathcal{Y}$.
Assume that $({\rm CoFib}_{\omega}, \ {\rm Fib}_{\omega}, \ {\rm Weq}_{\omega})$ is a model structure on  $\mathcal A$.
We need to prove that  $(\mathcal{X}, \mathcal{Y})$ is a hereditary complete cotorsion pair, and $\omega$ is contravariantly finite in $\mathcal A$.

\vskip5pt

By definition one easily sees that the class $\mathcal{C}_{\omega}$ of cofibrant objects is $\mathcal{X}$, and the class $\mathcal{F}_{\omega}$ of fibrant objects is $\mathcal{A}$.
Also,  the class $\mathcal{W}_{\omega}$ of trivial objects is $\mathcal{Y}$.
Indeed, for any $Y \in \mathcal{Y}$, since $(0, 0): Y\oplus 0\longrightarrow 0$ is a deflation with $0\in \omega$ and $\Ker (0, 0) = Y\in \Y$,
by definition $0: Y\longrightarrow 0$ a weak equivalence, i.e., $Y \in \mathcal{W}_{\omega}$; conversely, if $W \in \mathcal{W}_{\omega}$, i.e., $0: W\longrightarrow 0$ is a weak equivalence,
then there is a deflation $(0, 0): W \oplus W' \longrightarrow 0$
 with $W' \in \omega$ and $\Ker(0, 0) = W \oplus W' \in \mathcal{Y}$. It follows that $W \in \mathcal{Y}$.

\vskip5pt

Thus we have $\mathcal{C}_{\omega} = \X, \ \mathcal{F}_{\omega} =\mathcal A, \  \mathcal{W}_{\omega} = \Y, \ {\rm T}\mathcal{C}_{\omega} = \omega, \ {\rm T}\mathcal{F}_{\omega} = \mathcal{Y}$.

\vskip5pt

By the construction of $\CoFib_\omega$, any inflation with cokernel in $\mathcal C_\omega = \X$ belongs to $\CoFib_\omega$.  It follows from
Lemma \ref{extinclusion}(1) that $\Ext^1_\mathcal A(\X, \Y) = \Ext^1_\mathcal A(\mathcal{C}_{\omega}, {\rm T}\mathcal{F}_{\omega}) = 0$.
Thus, by Lemma \ref{tcofibandtfib} one has ${\rm TFib}_{\omega} = {\rm Fib}_{\omega}\cap \Weq_{\omega}.$

\vskip5pt

Hence both the conditions in Proposition \ref{ctpfrommodel}(1) are satisfied: cofibrations are precisely inflations with cofibrant cokernel and that any trivial fibration is a deflation.
It follows from Proposition \ref{ctpfrommodel}(1) that
$(\X, \Y) = (\mathcal{C}_{\omega}, {\rm T}\mathcal{F}_{\omega})$ is a complete cotorsion pair.

\vskip5pt

The heredity of the cotorsion pair $(\X, \Y)$ is guaranteed by Proposition \ref{resolvingcoresolving}.

\vskip5pt

By Lemma \ref{contravariantlyfinite}(3),  $\omega = {\rm T}\mathcal{C}_{\omega}$ is contravariantly finite in $\mathcal{A}$.
 \hfill $\square$

\section{\bf The correspondence of Beligiannis and Reiten}

\subsection{\bf Weakly projective model structures}

For a model structure on an exact category, keep the notations in Subsection 4.1. So $\mathcal C$ (respectively, $\mathcal F$, ${\rm T}\mathcal C$, and ${\rm T}\mathcal F$)
is the class of cofibrant objects (respectively, fibrant objects, trivially cofibrant objects, and trivially fibrant objects).

\begin{lem}\label{extvanish} \ Let  $(\CoFib ,\  \Fib ,\  \Weq )$ be a model structure on exact category $\mathcal A$.

\vskip5pt

$(1)$   \ If \ $\Ext^1_{\mathcal{A}} (\mathcal{C}, {\rm T}\mathcal{F}) = 0$ and any trivial fibration is a deflation, then any inflation with cofibrant cokernel is a cofibration.

\vskip5pt

$(2)$  \ If \ $\Ext^1_{\mathcal{A}} (\mathcal{C}, {\rm T}\mathcal{F}) = 0$ and any cofibration is an inflation,
then any deflation with trivially fibrant kernel is a trivial fibration.
\end{lem}

\begin{proof} \ We only justify (1); the assertion (2) can be similarly proved.

\vskip5pt

(1) \ Let $i: A\longrightarrow B$ be an inflation with $\Coker f\in \mathcal C$. Given an arbitrary trivial fibration $p$,
by assumption $p$ is a deflation.  By Lemma \ref{factorthrough}(3), $\Ker p\in \rm T\mathcal F$.
Since $\Ext^1_\mathcal A(\mathcal C, \rm T\mathcal F) = 0$, one can apply
the Extension-Lifting Lemma \ref{extlifting} to see that $i$ has
the left lifting property respect to $p$. Thus $i$ is a cofibration, by Proposition \ref{quillenlifting}.
\end{proof}

\begin{prop}\label{weaklyproj} \ Let  $(\CoFib ,\  \Fib ,\  \Weq )$ be a model structure on exact category $\mathcal A$. Then the following are equivalent.

\vskip5pt

$(1)$ \ Cofibrations are exactly inflations with cofibrant cokernel, and any trivial fibration is a deflation.

\vskip5pt

$(2)$   \ $\Ext^1_{\mathcal{A}} (\mathcal{C}, {\rm T}\mathcal{F}) = 0,$ any cofibration is an inflation,
and any trivial fibration is a deflation.

\vskip5pt

$(3)$ \ Trivial fibrations are exactly deflations with trivially fibrant kernel, and any cofibration is an inflation.

\vskip5pt

$(4)$ \ Cofibrations are exactly inflations with cofibrant cokernel, and trivial fibrations are exactly deflations with trivially fibrant kernel.

\vskip5pt

Moreover, if in addition $\mathcal A$ is weakly idempotent complete, then all the conditions above are equivalent to

\vskip5pt

$(5)$ \ $(\mathcal{C}, {\rm T}\mathcal{F})$ is a complete cotorsion pair.
\end{prop}
\begin{proof} \ The implication $(1) \Longrightarrow (2)$  follows from Lemma
\ref{extinclusion}(1).

\vskip5pt

$(2) \Longrightarrow (1)$:   By Lemma \ref{extvanish}(1), any inflation with cofibrant cokernel is a cofibration;
conversely, by assumption any cofibration $i$ is an inflation, and hence $\Coker i$ is cofibrant, by Lemma \ref{factorthrough}(4).
Thus, cofibrations are exactly inflations with cofibrant cokernel.

\vskip5pt

Similarly one can see $(2)\Longleftrightarrow (3)$.

\vskip5pt

$(4) \Longrightarrow (1)$ is clear; and $(1) \Longrightarrow (4)$ is also clear, since $(1)$ and $(3)$ imply $(4)$.

\vskip5pt

$(1) \Longrightarrow (5)$ follows from Proposition \ref{ctpfrommodel}(1).
It remains to prove  $(5) \Longrightarrow (2)$, if in addition $\mathcal A$ is weakly idempotent complete.

\vskip5pt

First we show that any cofibration is an inflation. Let $f: A\longrightarrow B$ be a cofibration. By the completeness of the cotorsion pair $(\mathcal{C}, {\rm T}\mathcal{F})$,
there is an inflation $i: A\longrightarrow Y$ where $Y\in {\rm T}\mathcal{F}$. By Lemma \ref{factorthrough}(2), $i$ factors through $f$.
Since $\mathcal A$ is weakly idempotent complete,  $f$ is an inflation. Similarly, any trivial fibration is a deflation. This completes the proof.
\end{proof}

Thus, the equivalent conditions in Proposition \ref{weaklyproj} are weaker than the conditions of an exact model structure.

\begin{defn}\label{defwp} \ A model structure on an exact category is  {\it weakly projective,} provided
that each object is fibrant and it satisfies the equivalent conditions in {\rm Proposition \ref{weaklyproj}.}
\end{defn}

\subsection{Proof of Theorem \ref{brcorrespondence}} By Theorem \ref{mainthm}, $\rm{Im}\Phi\in S_M$ and $\Psi\Phi={\rm Id}$.
It remains to prove $\rm{Im}\Psi\in S_C$ and $\Phi \Psi ={\rm Id}$.

\vskip5pt

For this purpose, let $(\CoFib, \Fib, \Weq)\in S_M$ be a  weakly projective model structure.
By Proposition \ref{ctpfrommodel}(1), $(\mathcal{C}, {\rm T}\mathcal{F})$ is a complete cotorsion pair.
Since $\mathcal{F} = \mathcal{A}$,  $\mathcal{C}\cap \rm T\mathcal{F} = {\rm T}\mathcal{C}\cap \mathcal{F} = {\rm T}\mathcal{C}$.
Thus,  by Lemma \ref{contravariantlyfinite}(3),  $\mathcal{C}\cap \rm T\mathcal{F} = {\rm T}\mathcal{C}$ is contravariantly finite in $\mathcal{A}$.

\vskip5pt

We need to prove that cotorsion pair $(\mathcal{C}, {\rm T}\mathcal{F})$ is hereditary (and hence $(\mathcal{C}, {\rm T}\mathcal{F})\in S_C$), and that
$(\CoFib, \Fib, \Weq) = (\CoFib_{\omega}, \Fib_{\omega}, \Weq_{\omega})$, where $\omega = \mathcal{C}\cap \rm T\mathcal{F}= {\rm T}\mathcal{C}$.
This will be done in several steps.

\vskip5pt

Since $(\CoFib, \Fib, \Weq)$ is a weakly projective model structure, by Proposition \ref{weaklyproj}(4) one has already $$\CoFib = \{\mbox{inflation} \ i \ | \ \Coker i\in \mathcal C\}= \CoFib_{\omega}$$
and $$\TFib= \{\mbox{deflation} \ p \ | \ \Ker p\in {\rm T}\mathcal F\} = \TFib_\omega.$$

\vskip5pt

{\bf Step 1:} \ $\TCoFib= \{\mbox{splitting monomorphism} \ f \ | \ \Coker f\in {\rm T}\mathcal C\} = \TCoFib_\omega$.

\vskip5pt

In fact, let $f: A\longrightarrow B$ be a splitting inflation with $\Coker f\in \rm T\mathcal{C}$. Then
there are morphisms $i: B\longrightarrow A$ and $p: \Coker f \longrightarrow B$ such that
$i\circ f= 1_A, \ \ \pi \circ p = {\rm Id}_{\Coker f}, \ \ i\circ p = 0$, where $\pi: B \longrightarrow\Coker f.$ Then it is clear that the square
\[
\xymatrix@R=0.8cm{
0\ar[d]\ar[r] &A\ar[d]^{f}\\
\Coker f\ar[r]^-p &B = A\oplus \Coker f
}
\]
is a pushout. Since $\Coker f$ is a trivially cofibrant object, $0\longrightarrow \Coker f\in \TCoFib$. It follows from Fact \ref{elementpropmodel}(3) that $f\in \TCoFib$.

\vskip5pt

Conversely, let $f: A\longrightarrow B$ be a trivial cofibration. Then  $f\in \CoFib$, and hence $f$ is an inflation. By Lemma \ref{factorthrough}(4), $\Coker f\in \rm T\mathcal{C}$. Since $A\in \mathcal{F}=\A$, it follows from  Lemma \ref{factorthrough}(2) that $1_A: A\longrightarrow A$ facts through $f$, i.e., $f$ is a splitting inflation.
This completes {\bf Step 1}.

\vskip5pt

{\bf Step 2:} \ $\Weq= \Weq_\omega$. This follows from
\ $\Weq=\TFib \circ \TCoFib
= \TFib_\omega\circ \TCoFib_\omega= \Weq_\omega$.

\vskip5pt

{\bf Step 3:} \ $\Fib = \{\mbox{morphism} \ p \ | \ p \ \mbox{is} \ \omega\mbox{-epic}\} = \Fib_\omega$.

\vskip5pt

In fact, by {\bf Step 1} and using the fact that $\Fib$ is precisely the class of morphisms which have the right lifting property with respect to all the trivial cofibrations
(cf. Proposition \ref{quillenlifting}(3)) one can easily see this: because that trivial cofibrations are splitting inflations with cokernel in $\rm T\mathcal C$,
and that a morphism $p$ has
the right lifting property with respect to trivial cofibrations is amount to say that $p$ is $\omega$-epic.

\vskip5pt

We have proved $\CoFib = \CoFib_\omega, \ \ \Fib = \Fib_\omega, \ \ \Weq = \Weq_\omega.$ Thus
$(\CoFib_{\omega}, \Fib_{\omega}, \Weq_{\omega})$ is also a model structure.
It follows from Proposition \ref{resolvingcoresolving} that cotorsion pair $(\mathcal{C}, {\rm T}\mathcal{F})$ is hereditary.
Thus $\rm{Im}\Psi\in S_C$ and $\Phi \Psi ={\rm Id}$.  This completes the proof. \hfill $\square$

\subsection{Model structures which are both exact and weakly projective}

An exact model structure on $\mathcal A$ is {\it projective} if each object is fibrant, or equivalently, the trivially cofibrant objects are projective. See \cite[4.5]{G}.
In this case $\mathcal A$ has enough projective objects.

\begin{cor} \label{theintersection} \ Let $\mathcal A$ be a weakly idempotent complete exact category. Then
a model structure on $\mathcal A$ is both exact and weakly projective if and only if it is projective.
If this is the case, then the left triangulated structure on ${\rm Ho}(\mathcal A)$ is in fact a triangulated category.

\end{cor}
\begin{proof} \ By definition a projective model structure is exact and each object is fibrant, thus it satisfies the equivalent conditions in Proposition \ref{weaklyproj}, and hence
it is weakly projective.
It remains to justify the last assertion. In this case the Hovey triple is of the form
$(\mathcal C, \ \A, \ \mathcal W)$. In particular,
$(\mathcal C\cap \mathcal W, \ \A)$ is a complete cotorsion pair in $\A$.
Thus $\mathcal A$ has enough projective objects and $\mathcal C\cap \mathcal W = \mathcal{P}$, the class of projective objects.  By Theorem \ref{brcorrespondence}
the complete cotorsion pair $(\mathcal C, \mathcal W)$ is hereditary.
Thus, the left triangulated structure on ${\rm Ho}(\mathcal A)$ is a triangulated category, by \cite[Theorem 6.21]{S}.\end{proof}

Recall that a complete cotorsion pair $(\X,\Y)$ is {\it generalized projective} (or gpctp, in short) if $\Y$ is thick and $\X\cap \Y=\mathcal P$, the class of projective objects.
See \cite[1.6, 7.11]{CRZ}, \cite[1.1.9]{Bec}.

\begin{cor} \label{gpctp} \ Let $\mathcal A$ be a weakly idempotent complete exact category,
$\mathcal{X}$ and $\mathcal{Y}$ full additive subcategories of $\mathcal{A}$ which are closed under isomorphisms and direct summands. Put $\omega:=\mathcal{X}\cap \mathcal{Y}$.
Then the following are equivalent.

\vskip5pt

$(1)$ \ $(\mathcal{X}, \mathcal{Y})$ is a gpctp, and $\mathcal A$ has enough projective objects$;$

\vskip5pt

$(2)$ \ \ $(\mathcal{X}, \mathcal{Y})$ a hereditary complete cotorsion pair, $\mathcal A$ has enough projective objects, and $\omega = \mathcal P;$

\vskip5pt

$(3)$ \ $(\CoFib_{\omega}, \Fib_{\omega}, \Weq_{\omega})$ is an exact model structure$;$

\vskip5pt

$(4)$ \ $(\CoFib_{\omega}, \Fib_{\omega}, \Weq_{\omega})$ is a projective model structure.

\end{cor}
\begin{proof} $(1) \Longrightarrow (2)$: Since $\Y$ is thick, $\Y$ is closed under the cokernel of inflations. Thus $(\X,\Y)$ is hereditary by Lemma \ref{hereditary}.

\vskip5pt

$(2)\Longrightarrow (1)$: Since $\mathcal A$ has enough projective objects and $\omega = \mathcal P,$ $\omega$  is contravariantly finite.
Thus $(\CoFib_{\omega}, \Fib_{\omega}, \Weq_{\omega})$ is an exact model structure, by Proposition \ref{thesame}. By Theorem \ref{mainthm}, $\Y$ is the class of trivial objects.
Thus $\Y$ is thick (cf. Theorem \ref{hoveycorrespondence}).

\vskip5pt

$(2) \Longrightarrow (4)$: By Theorem \ref{brcorrespondence}, $(\CoFib_{\omega}, \Fib_{\omega}, \Weq_{\omega})$ is a weakly projective model structure;
by Proposition \ref{thesame}, this model structure is exact; and then it is projective, by Corollary \ref{theintersection}.

\vskip5pt

$(4) \Longrightarrow (3)$ is clear.

\vskip5pt

$(3) \Longrightarrow (2):$ By Theorem \ref{mainthm}, $(\mathcal{X}, \mathcal{Y})$ a hereditary complete cotorsion pair; and then by Proposition \ref{thesame}
one knows that $\mathcal A$ has enough projective objects and $\omega = \mathcal P.$ \end{proof}

\subsection{\bf Final remarks: the dual version} \ For convenience, we state the dual version of the main results without proofs.
Let $\mathcal{A}$ be a weakly idempotent complete exact category, $\mathcal{X}$ and $\mathcal{Y}$ full additive subcategories of $\mathcal{A}$ which are closed under isomorphisms  and direct summands. Put $\omega=\mathcal{X}\cap \mathcal{Y}$.

\vskip5pt

Denote by ${\rm CoFib}^{\omega}$ the class of morphisms $f: A\longrightarrow B$ such that $f$ is $\omega$-monic, i.e.,
$\text{Hom}_{\mathcal{A}}(f, W): \Hom_{\mathcal{A}}(B,W)\longrightarrow \text{Hom}_{\mathcal{A}}(A,W)$ is surjective, for any object $W\in \omega$.

\vskip5pt

Denote by ${\rm Fib}^{\omega}$ the class of deflations $f$ with \ $\Ker f\in \mathcal{Y}$.

\vskip5pt

Denote by ${\rm Weq}^{\omega}$ the class of morphisms $f: A\longrightarrow B$ such that there is an inflation $\left(\begin{smallmatrix} f\\t \end{smallmatrix}\right): A \longrightarrow B\oplus W$ with $W\in \omega$ and $\Coker \left(\begin{smallmatrix} f\\t \end{smallmatrix}\right)\in \mathcal{X}$.

\begin{thm}\label{dualmainthm} \ Let $\mathcal A$ be a weakly idempotent complete exact category,
$\mathcal{X}$ and $\mathcal{Y}$ additive full subcategories of $\mathcal{A}$ which are closed under isomorphisms  and direct summands, and $\omega:=\mathcal{X}\cap \mathcal{Y}$.
Then $({\rm CoFib}^{\omega}, \ {\rm Fib}^{\omega}, \ {\rm Weq}^{\omega})$ is a  model structure on $\A$ if and only if $(\mathcal{X},\mathcal{Y})$ is a hereditary complete cotorsion pair in $\mathcal{A}$, and $\omega$ is covariantly finite in $\mathcal A$.

\vskip5pt

If this is the case, then the class $\TCoFib^\omega$ of trivial cofibrations is precisely
the class of inflations with cokernel in $\mathcal X$, and the class $\TFib^\omega$
of trivial fibrations is precisely the class of splitting epimorphisms with kernel in $\omega$; the class of cofibrant objects is $\mathcal A$, the class of fibrant objects is $\mathcal Y$,
the class of trivial objects is $\mathcal X;$ and the homotopy category is $\mathcal Y/\omega$.
\end{thm}

\vskip5pt

A model structure $(\CoFib, \Fib, \Weq)$ on $\mathcal A$ is {\it weakly injective} if
$\Fib$ is precisely the class of deflations with fibrant kernel, each trivial cofibration is an inflation, and each object is cofibrant.

\vskip5pt

\begin{thm} \label{dualcorrespondence} \ Let $\A$ be a weakly idempotent complete exact category. Denote by $S^C$ the class of hereditary complete cotorsion pairs $(\X,\Y)$ with $\omega=\X\cap \Y$ covariantly finite. Denote by $S^M$ the class of weakly injective model structures on $\A$. Then the maps $\Phi: (\X,\Y)\mapsto (\CoFib^{\omega}, \Fib^{\omega}, \Weq^{\omega})$ and $\Psi: (\CoFib, \Fib, \Weq)\mapsto (\rm T \mathcal{C}, \mathcal{F})$ give a bijection between $S^C$ and $S^M$, where $\rm T\mathcal C$ and $\mathcal{F}$ are respectively the class of trivially cofibrant objects and the class of fibrant objects.
\end{thm}

\vskip20pt

{\bf Acknowledgement}: The authors sincerely thank the anonymous referee for helpful comments and suggestions which improve the presentation of the paper.

\vskip20pt

\centerline {\bf Declaration of competing interest}

\vskip5pt

The authors declare that they have no known competing financial interests or personal relationships that could have appeared to influence the work reported in this paper.

\end{document}